 \documentclass[11pt]{amsart}
\usepackage[margin= 1.3 in]{geometry}

\usepackage{pdfsync}

 \usepackage[plainpages=false]{hyperref}
 \usepackage{amsfonts,amsmath,amssymb,amsthm,accents,mathtools}
 \usepackage{latexsym,lscape,rawfonts,mathrsfs}

 \usepackage[all]{xy}
 \usepackage{eufrak}
 \usepackage{graphicx,psfrag}
 \usepackage{pstool}

 \usepackage{array,tabularx}

 \usepackage{appendix}

\usepackage{txfonts}

 \newcommand{\ba}{\begin{align}}
 \newcommand{\ea}{\end{align}}
 \newcommand{\bal}{\begin{align*}}
 \newcommand{\eal}{\end{align*}}

 \DeclareMathOperator{\diam}{diam}

 \newcommand{\Rc}{\mathcal{R}c}
 \newcommand{\Sc}{\mathcal{R}}

 \newcommand{\dvol}{\text{d}V}

\renewcommand{\epsilon}{\varepsilon}

 \makeatletter
 \def\ExtendSymbol#1#2#3#4#5{\ext@arrow 0099{\arrowfill@#1#2#3}{#4}{#5}}
 
 \makeatother

 \makeatletter
 \def\ExtendSymbol#1#2#3#4#5{\ext@arrow 0099{\arrowfill@#1#2#3}{#4}{#5}}
 
 \makeatother

 \definecolor{hao}{rgb}{1,0.5,0}
 \definecolor{miao}{cmyk}{0.5,0,0.2,0.2}
 \definecolor{qiao}{gray}{0.96}

\newtheorem{prop}{Proposition}[section]

\newtheorem{proposition}[prop]{Proposition}

\newtheorem{theorem}[prop]{Theorem}

\newtheorem{lemma}[prop]{Lemma}

\newtheorem{corollary}[prop]{Corollary}

\newtheorem*{theorem*}{Theorem}
\theoremstyle{remark}
\newtheorem{remark}{Remark}

\numberwithin{equation}{section}

\title[Ricci flows with collapsing initial data (I)]{Notes on Ricci flows with 
collapsing time-slices (I):\\ Distance distortion}


\date{\today}
\author{Shaosai Huang}

\address{Department of Mathematics, University of Wisconsin - Madison, 480
Lincoln Drive, Madison, WI, 53706 U.S.A.}
\email{sshuang@math.wisc.edu}

\begin{document}
\maketitle

\begin{abstract}
In this note, we prove a uniform distance distortion estimate for Ricci flows
with uniformly bounded scalar curvature, independent of the lower bound of the
initial $\mu$-entropy. Our basic principle tells that once correctly renormalized, the
metric-measure quantities obey similar estimates as in the non-collapsing
case; espeically, the lower bound of the renormalized heat kernel, observed
on a scale comparable to the initial diameter, matches with the lower bound of
the renormalized volume ratio, giving the desired distance distortion estimate.
\end{abstract}
\section{Introduction}
For a fixed Ricci flow, a fundamental question of Richard Hamilton (see Section
17 of~\cite{Ham93}) is to obtain a uniform distance distortion estimate
depending on a minimal requirement of the space-time curvature bound. A
natural and non-trivial condition is to assume a uniform bound of the scalar
curvature in space-time, as evidenced by K\"ahler-Ricci flows on Fano manifolds.
The distance distortion problem in this case is completely settled by Chen-Wang
in~\cite{CW12}, and again in~\cite{CW14} as an important intermediate step
towards their main result. The K\"aher condition was then dropped by
Bamler-Zhang in~\cite{BZ15a}. See also the previous works of Richard
Hamilton~\cite{Ham93}, Miles Simon~\cite{Simon09} and Tian-Wang~\cite{TW15} for
several important partial results. However, all these estimates, including the
ones of Chen-Wang and Bamler-Zhang, rely on the uniform lower bound of the
initial $\mu$-entropy, a crucial condition that we will relax in this note.

As a second motivation, in studying the uniform behavior of \emph{all} Ricci
flows, one may have to encounter a family of Ricci flows without a uniform lower
bound for the initial $\mu$-entropy. A very common situation is when the family
of initial data have their diameter uniformly bounded, but volume degenerating
to 0, causing the initial $\mu$-entropy to approach negative infinity. A natural
question would then be whether there is a limiting metric space whose metric
evolves in a way determined by the Ricci flows (see Proposition~\ref{prop:
weak_compactness}). In this note, we make efforts towards this direction via the
following uniform distance distortion estimate along the Ricci flows:
\begin{theorem}
Let $(M,g(t))$ be a complete Ricci flow solution on $[0,T]$ with initial
diameter $D_0$ and initial volume $V$, and assume the following conditions:
\begin{enumerate}
 \item $(M, g(0))$, as a closed Riemannian manifold, has its doubling constant
 uniformly bounded above by $C_D$, and its $L^2$-Poincar\'e constant by
 $C_P$, and
  \item the scalar curvature is uniformly bounded in space-time: $\sup_{M\times 
  [0,T]}|\Sc_{g(t)}|\le C_R$.
\end{enumerate} 
There exist two positive constants
$\alpha=\alpha(\theta\ |\ C_D,C_P,C_R,D_0,n,T)<1$ with
\begin{align*}
\lim_{\theta\to 0}\ \alpha(\theta\ |\ C_D,C_P,C_R,D_0,n,T)= 0,
\end{align*}
 and $\nu=\nu(C_D,C_P,C_R,n)<1$, such that whenever $VD_0^{-n}\le \nu
 \omega_n$, for fixed $t\in [0,T]$ and $r\in (0,\sqrt{t})$, if we set
 $\theta:=\min\{1,r\slash D_0\}$, then 
\begin{align*}
\forall x,y\in M\ \text{with}\ d_{g(t)}(x,y)\ge r,\quad \text{and}\quad
\forall s\in (t-\alpha r^2,\min\{T,t+\alpha r^2\}),
\end{align*} 
 we have
\begin{align}
\alpha(\theta) d_{g(t)}(x,y)\ \le\  d_{g(s)}(x,y)\ \le\ 
\alpha(\theta)^{-1}d_{g(t)}(x,y).
\label{eqn: main}
\end{align}
\label{thm: main}
\end{theorem}
\begin{remark}
The requirement that $VD_0^{-n}<\nu \omega_n$ indicates that the initial
data is volume collapsing with bounded diameter. Notice that with
$\omega_n$ being the volume of the $n$-dimensional Euclidean unit ball, $\nu
\omega_n$ is a dimensional constant only depending on $C_D,C_P$ and $C_R$.
\end{remark}
In the statement of the theorem, $\theta$ refers to the relative size of
the scale on which we consider distance distortion compared to the initial
diameter, and $\alpha\approx \theta^{8n}e^{-\theta^{-2}}$. The bound $\alpha$
becomes worse as the scale on which we observe becomes smaller compared to the
initial diameter.
  
This is reasonable, as demonstrated in the case of collapsing initial data with
bounded curvature and diameter: there will be no uniform estimate of the distance
distortion in the fiber directions. However, we notice that such estimate is not
needed for providing a rough metric structure on the collapsing limit, since
eventually it is the estimates in the base directions that we will need.
Therefore, regardless of how small the relative scale we are considering, a
\emph{uniform} estimate, even though depending on such scale, is indeed what we
need.

The previous distance distortion estimates are based on the estimates of the
volume ratio change along the Ricci flow: with uniformly bounded scalar
curvature and initial entropy, the volume ratio at a point can neither suddenly
decrease (no local collapsing theorem of Perelman~\cite{Perelman02}), nor
suddenly increase (non-inflation property due to Chen-Wang~\cite{CW13} and Qi S.
Zhang~\cite{Zhang12}). Discretizing the geodesic distance by the number of
fix-sized geodesic balls that suitably cover the minimal geodesic, these
non-collapsing and non-inflation properties together provide the desired
control of the distance distortion. This type of ``ball containment'' argument
is succinctly discribed in the third section of Chen-Wang~\cite{CW16}.

In order to obtain uniform estimates of the change of volume ratio along the
Ricci flow, in the K\"ahler case Chen-Wang~\cite{CW14} studied the Bergman
kernel, while in the Riemannian case, Bamler-Zhang~\cite{BZ15a} relies on
Qi S. Zhang's heat kernel estimates in~\cite{Zhang12}.

Our theorem is proven along the same paths that lead to such estimates. However,
we need to start from scratch: underlying the estimates of the heat kernel, a
corner stone is the expression of the log-Sobolev constant in terms of
the initial $\mu$-entropy (see \cite{Ye15} and \cite{Zhang07}), which, in the
current note, will be replaced by a renormalized version involving the
\emph{initial global volume ratio} $VD_0^{-n}$.

Heuristically speaking, collapsing is a geometric phenomenon, while the behavior
of the heat kernel (which reflects the volume ratio) is analytic in nature. The
monotonicity of Perelman's functionals along the Ricci flow is
another instance where a geometric deformation bears an analytic meaning. A
basic principle in dealing with the analytic information associated with
collapsing, especially the Dirichlet energy and related objects, is
making a correct renormalization. This was first noticed by Kenji
Fukaya~\cite{Fukaya87b} in the setting of collapsing with bounded curvature and
diameter, and then strengthened through a series of work by Cheeger-Colding
(see \cite{ChCoI}, \cite{Cheeger99}, \cite{ChCoII} and \cite{ChCoIII}) to the
case with only Ricci curvature lower bound.

Our third and major motivation of this note is therefore to demonstrate the
necessity of the above renormalization principle in the setting of Ricci flows
with collapsing initial data: the initial collapsing is a geometric phenomenon,
yet in order to obtain the distance distortion estimate, we need to control the
analytic quantities --- the heat kernel bounds --- which could only be made
possible through a correct renormalization.

We now outline the series of estimates of the renormalized quantities that lead
to the uniform distance distortion estimate. We emphasize that these
inequalities are invariant under the parabolic rescaling of the Ricci flow, a
crucial point for them to work in a geometric setting. Also notice that the
constants involved are determined by $C_D,C_P,C_R,D_0,n$, but we only write
explicitly their dependence on $T$. Our starting point is a renormalized
$L^2$-Sobolev inequality (see \cite{Anderson92} and \cite{SC92}):
\begin{align}
\forall u\in H^1(M,g(0)),\quad \left(\int_Mu^{\frac{2n}{n-2}}\
\dvol_{g(0)}\right)^{\frac{n-2}{n}}\ \le\
C_S(VD_0^{-n})^{-\frac{2}{n}}\int_M|\nabla u|^2+D_0^2u^2\ \dvol_{g(0)},
\label{eqn: Sobolev}
\end{align}
where $D_0$ is the initial diameter and $V:=\int_M1\ \dvol_{g(0)}$ is the
initial volume.

Following classical arguments 
 and the
definition of the $\mathcal{W}$-functional, this gives a lower bound of the
initial entropy (\ref{eqn: initial_mu_lb}): for any $\tau>0$,
\begin{align*}
\mu(g(0),\tau)\ \ge\ \log
VD_0^{-n}-(C_RD_0^2+D_0^{-2})\tau-\frac{n}{2}\log (8n\pi eC_S).
\end{align*}
Here we would like to raise the readers' attention that it is not just the
initial total volume $V$, but the initial global volume ratio $VD_0^{-n}$, that
controls the lower bound of the entropy. This quantity not only technically
makes the inequality scaling-correct, but also conceptually reveals the meaning
of \emph{collapsing initial data} --- volume collapsing with bounded
diameter. 

Following Perelman's classical argument~\cite{Perelman02}, we could
deduce the lower bound of the renormalized volume ratio (see
Proposition~\ref{prop:
volume_ratio_lb}): there is a uniform $C_{VR}^-(T)>0$, such that
\begin{align*}
\forall t\in (0,T],\ \forall r\in (0,\sqrt{t}],\quad
(VD_0^{-n})^{-1}|B_{t}(x,r)|\ \ge\ C_{VR}^-(T)r^n.
\end{align*}
Here we start seeing the effect of the correct renormalization: even if the
volume ratio fails to have  a uniform lower bound, once renormalized by
$(VD_0^{-n})^{-1}$, it is indeed bounded below by $C_{VR}^-(T)$.\

Further exploring the definition and monotonicity of the
$\mathcal{W}$-functional, and following Qi S. Zhang's application~\cite{Zhang12} of
the method of Edward Davies~\cite{Davies89}, we obtain the following
rough upper bound of the renormalized heat kernel (see Proposition~\ref{prop:
heat_ub}): 
there is a uniform $C_H^+(T)>0$ such that
\begin{align*}
\forall t\in (0,T],\ \forall s\in (0,t),\ \forall
x,y\in M,\quad VD_0^{-n}G(x,s;y,t)\ \le\
C_H^+(T)(t-s)^{-\frac{n}{2}}.
\end{align*}
For the definition of $G(x,s;y,t)$ see Subsection 2.3. Here we see the duality
between the heat and the volume of a Riemannian manifold. Intuitively, the
collapsing is an intrinsic geometric procedure, and it should not cause the
addition or loss of the total heat. Therefore, if the global volume ratio
behavies like $VD_0^{-n}\to 0$, then the heat density should in general behave
like $(VD_0^{-n})^{-1}\to \infty$.

Up to this stage it is basically just the interplay between the Sobolev
inequality and the $\mathcal{W}$-functional: purely analytic in nature. In order
to estimate the distance distortion, we still need a lower bound of the
renormalized heat kernel. The original argument of Chen-Wang~\cite{CW13}
and Qi S. Zhang~\cite{Zhang12}, however, will not give us the desired bound: their
argument, based on the estimate of the reduced length of a space-constant curve
at the base point of the heat kernel, is valid regardless of scales; but in our
setting there is a drastic difference between the very small scales, which
resemble the locally $n$-dimensional Euclidean property of the manifold, and the
large scales, on which the collapsing to a lower dimensional space is
observed.

We will overcome this difficulty by obtaining a positive-time diameter bound in
terms of the initial diameter, and stick to our principle of keeping the
heat-volume duality. The following diameter bound is deduced following an
argument of Peter Topping in~\cite{Topping05} (see Proposition~\ref{prop:
diam_ub}): there exists a uniform constant $C_{diam}>0$ such that if the initial
global volume ratio is sufficiently small, i.e. $VD_0^{-n}<\nu \omega_n$ for
some uniform $\nu\in (0,1]$, then
\begin{align*}
\forall t\in (0,T],\quad \diam(M,g(t))\ \le\ C_{diam} e^{2C_Rt}D_0.
\end{align*}
This diameter bound is of great technical importance for us, since we will soon
use it to deduce an on-diagonal lower bound of the renormalized heat kernel.
Conceptually, this bound tells that scales that are comparable to the initial
diameter, remain comparable to the diameter at a positive time, up to a
uniform factor depending on the time elapsed.

At this stage, we could already prove some weak compactness result,
Proposition~\ref{prop: weak_compactness}, asserting the existence of a
Gromov-Hausdorff limit for positive time-slices --- recall that \emph{a
priorily}, we only assume a uniform scalar curvature bound on these time-slices.

With the help of the diameter bound above, we have the following lower bound of
the renormalized heat kernel (see Lemma~\ref{lem: heat_diagonal}): there exists
a uniform constant $C_H^-(T)>0$ and a positive function $\Psi(\theta\ |\ T)$
with $\lim_{\theta\to 0}\Psi(\theta\ |\ T)=0$, such that if $VD_0^{-n}\le \nu
\omega_n$,
\begin{align*}
\forall t\in (0,T],\ \forall s\in (0,t),\ \forall x\in M,\quad
VD_0^{-n}G(x,s;x,t)\ \ge\ C_{HD}^-(T)\Psi(\theta(s)|\ T)(t-s)^{-\frac{n}{2}}.
\end{align*}
Here we see that the effect of scales enters into the picture via the factor
$\Psi(\theta\ |\ T)$: for any $t\in (0,T]$ and any $s\in (0,t)$, $\theta(s) :=
\sqrt{t-s}\slash D_0$ is the ratio of the (parabolic) scale under consideration
compared to the initial diameter; when the scale that we observe approaches
$0$, relative to the initial diameter, then the lower bound of the renormalized
heat kernel will also approach $0$. (Rigorously speaking, we actually have
$\theta=\sqrt{t-s}\slash \diam(M,g(t))$ in our mind, but the diameter bound
above allows us to compare $r$ directly with $D_0$, making the definition more
canonical.) This estimate naturally leads to a Gaussian type lower bound of the
renormalized heat kernel, as well as the non-inflation property of the
renormalized volume ratio.

The bounds of the renormalized heat kernel, together with the previous
lower bound of the renormalized volume ratio, are enough to prove the desired
distance distortion estimate, in view of the arguments in proving Theorem
1.1 of~\cite{BZ15a}.

The current note consists of seven sections: We will start with recalling the
necessary background in Section 2. In section 3, we apply the renormalized
Sobolev inequality to obtain an initial entropy lower bound, explicitly
involving the initial global volume ratio. This will be used in the following
section to deduce a lower bound of the renormalized volume ratio, as well as an
upper bound of the positive-time diameter. In section 5, we obtain the bounds
of the renormalized heat kernel, and the proof of our main result is contained
in section 6. We will also discuss future work to be done in the final
section.

\subsection*{Acknowledgement}
I would like to thank Bing Wang for many useful discussions. I would also like
to thank Xiuxiong Chen, Yu Li and Selin Ta\c{s}kent for their interests in this
work.

\section{Background}
Our consideration will be on a closed Riemannian manifold $(M,g(0))$ whose
volume is $V$ and diameter is $D_0$. We assume that there exists a Ricci flow
up to time $T$, i.e. there is a family of smooth Riemannian metrics $g(t)$ on
$M$ satisfying the differential equation of symmetric two tensors:
\begin{align*}
\forall t\in [0,T],\quad \partial_tg\ =\ -2 \Rc_{g(t)}.
\end{align*}
We will also assume that the doubling constant of $(M,g(0))$ is given by $C_D$,
and its $L^2$-Poincar\'e constant by $C_P$. In this section, we will recall
the renormalized Sobolev inequality determined by $C_D$ and $C_P$, then
Perelman's $\mathcal{W}$-functional and $\mu$-entropy, and finally the gradient
estimates due to Bamler-Zhang. Instead of quoting directly the original
statements in the most general form, we will adapt these results in a
form that we could later make a direct use.

\subsection{The renormalized Sobolev inequality}
\addtocontents{toc}{\protect\setcounter{tocdepth}{1}}
We are inspired by the Sobolev constant estimate due to Michael
Anderson~\cite{Anderson92} (see also~\cite{GromovLevy}), in the situation where
a uniform Ricci curvature lower bound is assumed: for a fixed geodesic ball
$B(x,r)$, its Sobolev constant is comparable to
$(|B(x,r)|r^{-n})^{-\frac{2}{n}}$. The lesson is to consider
explicitly the effect of a correct renormalization, when applying the Sobolev
inequality to the study of Ricci flows.

In another direction, using methods in stochastic analysis and the Moser
iteration technique, Laurent Saloff-Coste has shown an even more general
Sobolev inequality in~\cite{SC92}, where the Sobolev constant only depends on
the doubling constant and the $L^2$-Poincar\'e constant. This is the inequality
that we will employ in this note:
\begin{proposition}[Renormalized $L^2$-Sobolev inequality]
Let $(M^n,g)$ be a Riemannian manifold such that the doubling constant and the
$L^2$-Poincar\'e constant are bounded from above by $C_D$ and $C_P$
respectively.
Then there is a constant $C_S=C_S(n,C_D,C_P)$ such that for any $B(x,r)\subset M$ and
any $u\in H^1_0(B(x,r))$, the following renormalized Sobolev inequality holds:
\begin{align}
\left(\int_{B(x,r)}u^{\frac{2n}{n-2}}\ \dvol_g\right)^{\frac{n-2}{n}}\ \le\ 
C_S(|B(x,r)|r^{-n})^{-\frac{2}{n}}\int_{B(x,r)}|\nabla
u|^2+r^{-2}u^2\ \dvol_g.
\label{eqn: local_Sobolev}
\end{align}
\end{proposition}
\begin{remark}
This is of course just one version of the Sobolev inequality. We call it
renormalized just to emphasize the independence of the Sobolev constant from the
volume, since eventually the volume will be sent to zero.
\end{remark}
The main point of this note is then to explore the geometric consequences of the
renormalization $(|B(x,r)|r^{-n})^{-\frac{2}{n}}$ in the setting of Ricci flows.
We notice that inequality (\ref{eqn: Sobolev}) is just a (weaker) global version of
this inequality.

\subsection{Perelman's $\mathcal{W}$-entropy}
\addtocontents{toc}{\protect\setcounter{tocdepth}{1}}
As mentioned in the introduction, the monotonicity of Perelman's
$\mathcal{W}$-functional along the Ricci flow is an instance where a
geometric deformation bears an analytic meaning. This connection is the
foundation of the current note. We now recall Perelman's
$\mathcal{W}$-functional~\cite{Perelman02}:
for any $\bar{t}\in (0,T]$, any $v^2\in C^1(M,g(\bar{t}))$ and any $\tau>0$,
\begin{align}
\begin{split}
\mathcal{W}(g(\bar{t}),v^2,\tau)\ :=\ 
\int_{M}\tau\left(4|\nabla
v|^2+\Sc_{g(\bar{t})}v^2\right)-v^2\log v^2-n\left(1+\frac{1}{2}\log(4\pi
\tau)\right)v^2\dvol_{g(\bar{t})}.
\end{split}
\label{eqn: entropy_defn}
\end{align}
If we require $\int_Mv^2\ \dvol_{g(\bar{t})}=1$, let $\tau$ solve $\tau'(t)=-1$,
and let $u(t)$ solve the conjugate heat equation along the Ricci flow:
$(\partial_t+\Delta-\Sc)u=0$ with the prescribed final data $u(\bar{t}):=v^2$,
then we have the monotone increasing property of the $\mathcal{W}$-functional:
\begin{align*}
\frac{\text{d}}{\text{d}t}\mathcal{W}(g(t),u(t),\tau)\ \ge\ 0.
\end{align*}

The $\mu$-entropy is defined as 
\begin{align*}
\mu(g(t),\tau)\ :=\ \inf_{\int_Mv^2\
\dvol_{g(t)}=1}\mathcal{W}(g(t),v^2,\tau),
\end{align*} 
and letting the data varying similarly as in the $\mathcal{W}$-functional, we
also obtain the monotone increasing property of the $\mu$-entropy.
\subsection{Heat equation solutions coupled with the Ricci flow}
\addtocontents{toc}{\protect\setcounter{tocdepth}{1}}

In this subsection we collect some point-wise estimates of heat equation
solutions coupled with the Ricci flow. For any $x,y\in M$ and $0\le s<t<T$, we
will let $G(x,s;y,t)$ denote the heat kernel coupled with the Ricci flow based
at $(x,s)$, i.e.
fixing $(x,s)\in M\times [0,T)$, we have
\begin{align}
(\partial_t-\Delta_{g(t)})G(x,s;-,-)\ =\ 0,\quad\text{and}\quad
\lim_{t\downarrow s}G(x,s;-,-)\ =\ \delta_{(x,s)},
\end{align}
where $\delta_{(x,s)}$ is the space-time Dirac delta function at $(x,s)\in
M\times [0,T)$.  On the other hand, fixing $(y,t)\in M\times (0,T)$ and setting
$(x,s)$ free, this same function satisfies 
\begin{align}
(\partial_s+\Delta_{g(s)}+\Sc_{g(s)})G(-,-;y,t)\ =\ 0,\quad\text{and}\quad
\lim_{s\uparrow t}G(-,-;y,t)\ =\ \delta_{(y,t)},
\end{align}
i.e. $G(-,-;y,t)$ is the conjugate heat kernel coupled with the Ricci flow based
at $(y,t)$.

Our heat kernel lower bound of Gaussian type will be base on the following
key gradient estimate due to Qi S. Zhang, see Theorem 3.3 in~\cite{Zhang06}:
\begin{proposition}[Gradient estimate]
Let $(M,g(t))$ be a Ricci flow on a complete $n$-manifold $M$ over time $[0,T)$
and let $u\in C^{\infty}(M\times [0,T))$ be a positive solution to the heat
equation $(\partial_t-\Delta)u=0$, $u(\cdot,0)=u_0$ coupled with the Ricci
flow. Then there is a constant $B<\infty$ depending only on $n$, such that if
$u\le a$ on $M\times [0,T]$ for some constant $a>0$, then $\forall (x,t)\in
M\times (0,T]$,
\begin{align}
\frac{|\nabla u|(x,t)}{u(x,t)}\le
\sqrt{\frac{1}{t}}\sqrt{\log\frac{a}{u(x,t)}}.
\end{align}
\label{thm: gradient_estimate}
\end{proposition}

Note that this inequality also reads
\begin{align*}
\left|\nabla \sqrt{\log\frac{a}{u}}\right|(x,t)\le \frac{1}{\sqrt{t}}
\end{align*}
for any $(x,t)\in M\times (0,T]$.

Now for any fixed $(x,t_0)\in M\times [0,T)$, let $G(x,t_0;-,-)$ be the coupled
heat kernel described above. Viewing $u(y,s)=G(x,t_0;y,s)$ as a coupled heat
equation solution on $M\times [\frac{t_0+t}{2},t]$, and integrating the above
inequality along minimal geodesics, we could get a Harnack inequality for heat
equation solutions coupled with the Ricci flow, also see inequality (3.44)
of~\cite{Zhang06}:
\begin{corollary}
We have $\forall (y,t),(y',t)\in M\times (t_0,T]$, 
\begin{align}
G(x,t_0; y,t)\ \le\
H(n)\left(\sup_{M\times
[(t_0+t)\slash 2,t]}G(x,t_0;-,-)\right)^{\frac{1}{2}}
G(x,t_0;y',t)^{\frac{1}{2}}e^{H'(n)d_{t}(y,y')^2\slash(t-t_0)},
\label{eqn: Harnack}
\end{align}
where $H(n)$ and $H'(n)$ are dimensional constants.
\end{corollary}

In order to estimate the distance distortion we also need a time derivative
bound of the coupled heat kernel. This is achieved by the following estimate,
which is Lemma 3.1(a) in~\cite{BZ15a}:
\begin{proposition}
Let $(M,g(t))$ be a Ricci flow on a closed $n$-manifold $M$ over time $[0,T]$
and let $u\in C^{\infty}(M\times [0,T])$ be a positive solution to the heat
equation $(\partial_t-\Delta)u=0$, $u(\cdot,0)=u_0$ coupled with the Ricci
flow. Then there is a constant $B<\infty$ depending only on $n$, such that if
$u\le a$ on $M\times [0,T]$ for some constant $a>0$, then $\forall (x,t)\in
M\times (0,T]$,
\begin{align*}
\left(|\Delta u|+\frac{|\nabla u|^2}{u}-a\Sc\right)(x,t)\le \frac{aB(n)}{t}.
\end{align*}
\end{proposition}

Again, setting $u(y,t)=G(x,t_0;y,t)$ for $t>t_0$, and
considering it as a coupled heat equation solution on $M\times
(\frac{t_0+t}{2},t)$, we immediately obtain
\begin{align}
\begin{split}
&|\partial_tG(x,t_0;y,t)|+\frac{|\nabla_yG(x,t_0;y,t)|^2}{G(x,t_0;y,t)}\\
 \le\quad &\sup_{M\times [(t_0+t)\slash 2,t]}G(x,t_0;-,-)\
\left(\Sc_{g(t)}(x,t_0;y,t)+\frac{B(n)}{t-t_0}\right).
\end{split}
\label{eqn: time_derivative}
\end{align}



\section{A uniform renormalized Sobolev inequality along the Ricci flow}

A uniform Sobolev inequality along Ricci flows will enable us to do analysis on
positive time slices. Notice that the lower bound of the $\mu$-entropy reflects
the upper bound of the log-Sobolev constant, and the monotone increasing
property of the $\mu$-entropy will further preserve, rather than destroying, the
log-Sobolev constant. In this section, we will see that the information of
initial global volume ratio is encoded in the initial $\mu$-entropy via a
log-Sobolev inequality, deduced following a classical argument, but with the
renormalized Sobolev inequality (\ref{eqn: Sobolev}) as our starting point. We
will also deduce a uniform renormalized Sobolev inequality along the Ricci
flow, which clearly shows how the initial global volume ratio affects the
Sobolev constants on positive time slices. For previous results we refer
the readers to the works of Rugang Ye~\cite{Ye15} and Qi S.
Zhang~\cite{Zhang07}, \cite{Zhang07Err}, \cite{Zhang07Add}.

\subsection{Lower bound of initial entropy via the renormalized
Sobolev inequality}
\addtocontents{toc}{\protect\setcounter{tocdepth}{1}}

From (\ref{eqn: Sobolev}), we see that if $\int_Mv^2\dvol_{g(0)}=1$, then 
\begin{align*}
\left(\int_M v^{\frac{2n}{n-2}}\dvol_{g(0)}\right)^{\frac{n-2}{n}}\le
4C_SV^{-\frac{2}{n}} \left(D_0^2\int_{M}|\nabla
v|^2\ \dvol_{g(0)}+1\right).
\end{align*}
Due to the uniform bound of the scalar curvature, we could further obtain
\begin{align*}
\left(\int_Mv^{\frac{2n}{n-2}}\ \dvol_{g(0)}\right)^{\frac{n-2}{n}}\le
4C_SV^{-\frac{2}{n}}\left(D_0^2\int_{M}\left(4|\nabla
v|^2+\Sc_{g(0)}v^2\right)\ \dvol_{g(0)}+C_RD_0^2+1\right).
\end{align*}

Since the logarithm function is concave, and since $v^2\dvol_{g(0)}$ defines a
probability measure on $M$, by Jensen's inequality, we have 
\begin{align*}
\forall u\in L^1(M,v^2\dvol_{g(0)}),\quad \int_M (\log |u|)\ v^2\dvol_{g(0)}\le
\log \int_M |u|\ v^2\dvol_{g(0)}.
\end{align*}
With $u=v^{q-2}$ with $q=\frac{2n}{n-2}$ (notice that
$\frac{q}{q-2}=\frac{n}{2}$), the above inequality gives
\begin{align*}
\begin{split}
\int_Mv^2\log v^2\ \dvol_{g(0)}\ =\ &\int_M \frac{2}{q-2}\left(\log
v^{q-2}\right)\ v^2\dvol_{g(0)}\\
\le\ &\frac{2}{q-2}\log \int_Mv^q\ \dvol_{g(0)},
\end{split}
\end{align*}
which is exactly $\frac{n}{2}\log\|v\|_{L^q(M)}^2$; furthermore,
\begin{align*}
\begin{split}
&\frac{n}{2}\log \|v\|_{L^q(M)}^2\\
\le\ &\frac{n}{2}\log\left(D_0^2\int_M\left(4|\nabla
v|^2+\Sc_{g(0)}v^2\right)\ \dvol_{g(0)}+C_RD_0^2+1\right)-\log V+\frac{n}{2}\log
4C_S.
\end{split}
\end{align*}
Now applying the elementary inequality $\log u \le \alpha u-1-\log \alpha$
for all $\alpha>0$ to the first term in the right-hand side of this last
inequality, we obtain
\begin{align}
\begin{split}
\int_Mv^2\log v^2\ \dvol_{g(0)}\ 
\le\ &\frac{n}{2}\log\left(D_0^2\int_M\left(4|\nabla
v|^2+\Sc_{g(0)}v^2\right)\ \dvol_{g(0)}+C_RD_0^2+1\right)\\
&-\log V+\frac{n}{2}\log
4C_S\\
\le\ &\frac{\alpha n}{2}\left(D_0^2\int_M\left(4|\nabla
v|^2+\Sc_{g(0)}v^2\right)\ \dvol_{g(0)}+C_RD_0^2+1\right)\\
&-\log V+\frac{n}{2}\left(\log 4C_S-1-\log \alpha\right).
\end{split}
\label{eqn: log_Sobolev_0}
\end{align}

Recalling the definition of the $\mathcal{W}$-functional (\ref{eqn:
entropy_defn}), and taking $\alpha=\frac{2\tau}{nD_0^2}$ in (\ref{eqn:
log_Sobolev_0}), we immediately see
\begin{align}
\mathcal{W}(g(0),v^2,\tau)\ \ge\ \log
VD_0^{-n}-(C_RD_0^2+D_0^{-2})\tau-\frac{n}{2}\log (8n\pi
eC_S).
\label{eqn: initial_W_lb}
\end{align}
Here $\tau$, as a multiple of $\alpha$, could be any positive number.
Since this is valid for any function $v$ on $M$ with unit $L^2$-norm, we have,
for any $\tau\in [T,2T]$,
\begin{align} 
\mu(g(0),\tau)\ \ge\ \log
VD_0^{-n}-(C_RD_0^2+D_0^{-2})\tau-\frac{n}{2}\log (8n\pi
eC_S).
\label{eqn: initial_mu_lb}
\end{align}
Here we notice that both sides of the inequalities above are invariant under a
parabolic rescaling.
\begin{remark}
It is well-know that collapsing initial data implies that there is no uniform
lower bound of the $\mathcal{W}$-entropy, and here we give an explicit lower
bound in terms of the initial global volume ratio.
\end{remark}

Now suppose we evolve $v^2$ at some $\bar{t}$-slice backward by the conjugate
heat equation, i.e. we consider a function $u$ such that
\begin{align*}
u(\bar{t})\ =\ v^2;\quad (\partial_t+\Delta_{g(t)}-\Sc_{g(t)})u=0;\quad
\partial_t g\ =\ -2\Rc_{g(t)},
\end{align*}
then $\mathcal{W}(g(t),u(t),\tau(t))$ is increasing in $t$ where $\tau'=-1$.
Therefore, for any $v^2$ with unit $L^1(g(\bar{t}))$-norm, we have, by the
monotone increasing property of the $\mathcal{W}$-functional, that
\begin{align*}
\mathcal{W}(g(\bar{t}),v^2,\tau(\bar{t}))\ \ge\
&\mathcal{W}(g(0),u(0),\tau(\bar{t})+\bar{t})\\
\ge\ &\log VD_0^{-n}-(C_R+D_0^{-2})(\tau(\bar{t})+\bar{t})-\frac{n}{2}\log
(8n\pi eC_S),
\end{align*}
or in the form of the log-Sobolev inequality,
\begin{align}
\begin{split}
\int_M\tau(4 |\nabla v|^2+\Sc_{g(\bar{t})}v^2)-v^2\log v^2\ \dvol_{g(\bar{t})}\
\ge\ \log
V\left(\frac{\tau}{D_0^2}\right)^{\frac{n}{2}}-(C_R+D_0^{-2})(\tau+\bar{t})-C_{lS},
\end{split}
\label{eqn: uniform_log_Sobolev}
\end{align}
where $C_{lS}:=\frac{n}{2}\log (2ne^{-1}C_S)$ and $\tau$ is
any positive number.

\subsection{Uniform renormalized Sobolev inequality along the Ricci flow}
\addtocontents{toc}{\protect\setcounter{tocdepth}{1}}

In this subsection, we establish a uniform renormalized Sobolev inequality
along the Ricci flow. We will follow the exposition of~\cite{Zhang11}, which is
based on the argument of Edward Davies~\cite{Davies89} in the case of a fixed
Riemannian manifold. The result of this subsection will not be needed in our
estimate of the distance distortion, yet we still include it here because we
will later use a similar argument to prove a rough upper bound of the renormalized
heat kernel in Section 5.1.

The first step would be using the uniform log-Sobolev inequality (\ref{eqn:
uniform_log_Sobolev}) to obtain an upper bound of the heat kernel on a
fixed future time slice $(M,g(\bar{t}))$. Now let $u$ be any
solution to the equation 
\begin{align*}
(\partial_t-\Delta_{g(\bar{t})}+\Sc_{g(\bar{t})})u\ =\ 0
\end{align*}
on the \emph{fixed} Riemannian manifold $(M,g(\bar{t}))$. Consider for any fixed
$t>0$, the exponent $p(s):=\frac{t}{t-s}$ for $s\in [0,t]$. We immediately
see that $p'(s)=t(t-s)^{-2}>0$, moreover,
\begin{align}
\begin{split}
0\ \le\ \frac{p(s)-1}{p'(s)}\ &=\ \frac{s(t-s)}{t}\ \le\ \frac{t}{4},\\
\text{and}\quad \frac{p'(s)}{p^2(s)}\ &=\ \frac{1}{t}.
\end{split}
\label{eqn: ps}
\end{align}
We also let
\begin{align*}
v(x,s)\
:=\ u(x,s)^{\frac{p(s)}{2}}\|u^{\frac{p(s)}{2}}\|^{-1}_{L^2(M,g(\bar{t}))},
\end{align*}
so that $\|v\|_{L^2(M,g(\bar{t}))}=1$. Routine computations give
\begin{align*}
\begin{split}
&p^2(s)\partial_s\log \|u\|_{L^{p(s)}(M,g(\bar{t}))}\\
=\ &p'(s)\int_{M}v^2\log
v^2\ \dvol_{g(\bar{t})}-4(p(s)-1)\int_M|\nabla
v|^2\ \dvol_{g(\bar{t})}-p^2(s)\int_M\Sc_{g(\bar{t})}v^2\ \dvol_{g(\bar{t})}\\
\le\ &p'(s)\left(\int_{M}v^2\log
v^2\dvol_{g(\bar{t})}-\frac{(p(s)-1)}{p'(s)}\int_M4|\nabla
v|^2+\Sc_{g(\bar{t})}v^2\
\dvol_{g(\bar{t})}+\frac{3t}{4}C_R\right).
\end{split}
\end{align*}
Thus if we plug $\tau=\frac{p(s)-1}{p'(s)}$ into (\ref{eqn:
uniform_log_Sobolev}), then the above computation, together with (\ref{eqn:
ps}) give
\begin{align*}
&\partial_s\log \|u\|_{L^{p(s)}(M,g(\bar{t}))}\\
\le\ &\frac{1}{t} \left(-\frac{n}{2}\log \frac{s(t-s)}{t}-\log VD_0^{-n}
+(C_R+D_0^{-2})\left(\bar{t}+\frac{s(t-s)}{t}\right)+C_{lS}+\frac{3t}{4}C_R\right).
\end{align*}
Notice that $p(0)=1$ and $p(t)=\infty$, we integrate the above inequality
(with respect to $s$) from $0$ to $t$ to obtain for any $t>0$,
\begin{align}
\log\frac{\|u(-,t)\|_{L^{\infty}(M,g(\bar{t}))}}{\|u(-,t)\|_{L^1(M,g(\bar{t}))}}\
\le\ -\frac{n}{2}\log t-\log VD_0^{-n}+(C_R-D_0^{-2})t+\tilde{C}_H(\bar{t}),
\end{align}
where $\tilde{C}_H(\bar{t})=2(\bar{t}+1)(C_R+D_0^{-2})+C_{lS}+n$.
Now let $G_{\bar{t}}(x,t,y)$ be the heat kernel of $(M,g(\bar{t}))$ centered at
$x\in M$, then 
\begin{align}
\begin{split}
&u(x,t)=\int_{M}G_{\bar{t}}(x,t,y)u(y,0)\ \dvol_{g(\bar{t})}(y),\\
\text{and}\quad
&\|u(-,t)\|_{L^1(M,g(\bar{t}))}=\int_{M}u(y,t)\ \dvol_{g(\bar{t})}(y),
\end{split}
\end{align}
we conclude that
\begin{align}
(VD_0^{-n})G_{\bar{t}}(x, t,y)\ \le\ e^{\tilde{C}_H(\bar{t})+(2C_R+D_0^{-2})t}
t^{-\frac{n}{2}}.
\label{eqn: heat_ub_0}
\end{align}

Now consider
$\tilde{G}_{\bar{t}}(-,t,-)
:=e^{-(2C_R+D_0^{-2})t}G_{\bar{t}}(-,t,-)$, then $\tilde{G}_{\bar{t}}$
is the fundamental solution to the equation
\begin{align*}
\left(\partial_t-\Delta_{\bar{t}}+\Sc_{g(\bar{t})}+(2C_R+D_0^{-2})\right)u\ =\
0,
\end{align*}
and by(\ref{eqn: heat_ub_0}) we have the control 
\begin{align*}
\forall t>0,\quad \tilde{G}_{\bar{t}}(-,t,-)\ \le\ 
\tilde{C}_H(\bar{t})(VD_0^{-n})^{-1}t^{-\frac{n}{2}}.
\end{align*}
Notice that $\tilde{C}_H(\bar{t})= 2\bar{t}(C_R+D_0^{-2})+C_{lS}+n$ is
independent of time and space variables on the fixed manifold $(M,g(\bar{t}))$;
also notice that it is invariant under the parabolic rescaling.

Now we can conclude that the operator of integrating against the kernel
$\tilde{G}_{\bar{t}}$ is a contraction, and standard argument gives the
$L^2$-Sobolev inequality on $(M,g(\bar{t}))$:
\begin{align}
\|f\|_{L^{\frac{2n}{n-2}}(M,g(\bar{t}))}^2\le
C_{Sob}(\bar{t})V^{-\frac{2}{n}}D_0^2\left(\|\nabla
f\|_{L^2(M,g(\bar{t}))}^2+(2C_R+D_0^{-2})\|f\|_{L^2(M,g(\bar{t}))}^2\right),
~\label{eqn:uniform_Sobolev}
\end{align}
where $C_{Sob}(\bar{t})=(2(\bar{t}+1)(C_R+D_0^{-2})+C_{lS}+n)^{\frac{2}{n}}$ is
uniformly bounded for bounded $\bar{t}$, independent of $V$ and the flow.

\section{Estimating the geometric quantities along the Ricci flow}
In this section we give a lower bound of the renormalized volume ratio on any
scale, and a scaling invariant upper bound of the diameter along the Ricci flow.
The estimates only depend on the initial doubling constant $C_D$, the initial
$L^2$-Poincar\'e constant $C_P$, the initial diameter $D_0$, the space-time
scalar curvature bound $C_R$, and the time elapsed from the beginning.

Both estimates are based on the idea that the $\mathcal{W}$-functional, when
tested against a suitable spacial cut-off function, bounds from below the volume
ratio at the given time slice, and then the monotone increasing property of the
$\mathcal{W}$-functional further provides the desired renormalization by the
initial total volume, as shown in (\ref{eqn: initial_W_lb}).

More specifically, throughout this section, we fix a time slice $\bar{t}\in
(0,T]$ and a scale $r$ such that $r^2\in (0,\bar{t}]$. For any fixed $x\in M$,
we could define a spacial cut-off function as 
\begin{align*}
h^2(y)\ =\ e^{-A}(4\pi
r^2)^{-\frac{n}{2}}\eta^2\left(r^{-1}d_{\bar{t}}(x,y)\right),
\end{align*}
 with $\eta$ being a smooth cut-off function supported on $[0,1)$, constantly
 equal to $1$ on $[0,\frac{1}{2}]$ and $-2\le \eta'\le 0$ on $(\frac{1}{2},1)$.
 Moreover, $A$ is chosen so that $\int_{M}h^2\ \dvol_{g(\bar{t})}=1$, and we
 immediately see 
\begin{align}
\frac{|B_{\bar{t}}(x,\frac{r}{2})|}{(4\pi r^2)^{\frac{n}{2}}}\ \le\ e^A\ =\ 
\int_{M}\frac{\eta^2\left(r^{-1}d_{\bar{t}}(x,y)\right)}{(4\pi
r^2)^{\frac{n}{2}}}\ \dvol_{g(\bar{t})}(y)\ \le\ 
\frac{|B_{\bar{t}}(x,r)|}{(4\pi r^2)^{\frac{n}{2}}}.
\label{eqn: estimate_eA}
\end{align}

Recall that the $\mathcal{W}$-functional for $(M,g(\bar{t},h^2)$ is
defined as
\begin{align*}
\mathcal{W}(g(\bar{t}),h^2,r^2)\ =\ \int_{M}4r^2|\nabla
h|^2+r^2\Sc_{g(\bar{t})}h^2-h^2\log h^2\
\dvol_{g(\bar{t})}-\frac{n}{2}\log (4\pi r^2)-n.
\end{align*}
We now roughly estimate some terms of the right-hand side of this inequality: 

Since $|\nabla d_{\bar{t}}|\le 1$, we have
\begin{align}
\begin{split}
\int_{M}4r^2|\nabla h|^2\ \dvol_{g(\bar{t})}\
\le\ &\int_{B_{\bar{t}}(x,r)}\frac{16r^2e^{-B}\left|\eta'\nabla
d_{\bar{t}}(x,y)\right|^2}{(4\pi
r^2)^{\frac{n}{2}}r^2}\ \dvol_{g(\bar{t})}(y)\\
\le\ &\frac{64|B_{\bar{t}}(x,r)|}{e^A(4\pi r^2)^{\frac{n}{2}}}\\
\le\ &\frac{64 |B_{\bar{t}}(x,r)|}{|B_{\bar{t}}(x,\frac{r}{2})|}
\end{split}
\label{eqn: estimate_gradient}
\end{align}
where we have used (\ref{eqn: estimate_eA}); moreover, since $h^2$ is
supported in $B_{\bar{t}}(x,r)$, and since the mapping $\sigma \mapsto -\sigma
\log \sigma$ is concave, we apply this to $\sigma =h^2$ and use Jensen's inequality
see
\begin{align}
\begin{split}
&\int_M-h^2\log h^2\ \dvol_{g(\bar{t})}-\frac{n}{2}\log 4\pi r^2-n\\
\le\ &-\int_{B_{\bar{t}(x,r)}}h^2\ \dvol_{g(\bar{t})}\ \left(\log
\fint_{B_{\bar{t}}(x,r)}h^2\ \dvol_{g(\bar{t})}\right)-\frac{n}{2}\log 4\pi
r^2-n\\
=\ &\log\left(|B_{\bar{t}}(x,r)|r^{-n}\right)-\frac{n}{2}\log 4\pi e^{2}.
\end{split}
\label{eqn: estimate_ulogu}
\end{align}

\subsection{Lower bound of the renormalized volume ratio}
\addtocontents{toc}{\protect\setcounter{tocdepth}{1}}
It is well known, as Perelman's no local collapsing theorem tells, that the
lower bound of the initial $\mu$-entropy and the upper bound of the scalar
curvature together give a lower bound of the volume ratio, see~\cite{Perelman02}
and~\cite{KL08}. Following this classical argument, but with the more explicit
lower bound (\ref{eqn: initial_mu_lb}) of the initial $\mu$-entropy, we obtain a
generalized lower bound of the renormalized volume ratio. We begin with the
following lemma:
\begin{lemma}
For the fixed time slice $\bar{t}$ and any positive $r\le \sqrt{\bar{t}}$,
suppose the doubling property
\begin{align}
|B_{\bar{t}}(x,\frac{r}{2})|\ \ge\ 
3^{-n}\left|B_{\bar{t}}(x,r)\right|
\label{eqn: local_doubling}
\end{align}
holds, then there is a constant $C_{VR}^-(T)=C_{VR}^-(T)(C_R,C_S,D_0,T)$ such
that
\begin{align}
\frac{|B_{\bar{t}}(x,r)|}{r^n}\ \ge\ C_{VR}^-(T) VD_0^{-n}.
\label{eqn: VR_lb}
\end{align}
\end{lemma}
\begin{proof}
We examine the upper bound of $\mathcal{W}(g(\bar{t}), h^2,r^2)$ with the help
of (\ref{eqn: local_doubling}). 

Since $\sup_{M\times [0,2T]}|\Sc_{g(t)}|\le
C_R$, we have
\begin{align}
\int_Mr^2\Sc_{g(\bar{t})}h^2\ \dvol_{g(\bar{t})}\ \le\ 2C_RT;
\label{eqn: estimate_Sc}
\end{align}
moreover, from (\ref{eqn: estimate_gradient}) and (\ref{eqn: local_doubling}) we
have
\begin{align*}
\int_M4r^2|\nabla h|^2\ \dvol_{g(\bar{t})}\ \le\ 3^{n+4}.
\end{align*} 
These estimates, together with (\ref{eqn: estimate_ulogu}) give
\begin{align}
\mathcal{W}(g(\bar{t}),h^2,r^2)\ \le\ \log
\frac{|B_{\bar{t}}(x,r)|}{r^n}+2C_RT+3^{n+5}-\frac{n}{2}\log 4\pi e^2.
\label{eqn: estimate_Wtbar}
\end{align}

On the other hand, since $\|h\|_{L^2(M,g(\bar{t}))}=1$, we could evolve
$h^2$ by the conjugate heat equation along the Ricci flow
$\partial_t g(t)=-2\Rc_{g(t)}$, i.e. we solve $(\partial_t+\Delta-\Sc)u=0$ with
final value $u(\bar{t})=h^2$.

By the monotone increasing property of $\mathcal{W}(g(t),u(t),\tau)$ in $t$
(with $\tau'(t)=-1$), we may apply the initial lower bound (\ref{eqn:
initial_W_lb}) to see
\begin{align*}
\mathcal{W}(g(\bar{t}),h^2,r^2)\ \ge\ &\mathcal{W}(g(0),u(0),\bar{t}+r^2)\\
\ge\ &\log VD_0^{-n}-\frac{(C_RD_0^2+1)}{D_0^2}(\bar{t}+r^2) -\frac{n}{2}\log
(8n\pi eC_SD_0^2),
\end{align*}
therefore by (\ref{eqn: estimate_Wtbar}), we have the following lower bound of
the log volume ratio:
\begin{align*}
\log\frac{|B_{\bar{t}}(x,r)|}{r^n}\ \ge\ 
\log VD_0^{-n}-\frac{2T(2C_RD_0^2+1)}{D_0^2}-3^{n+5}
-\frac{n}{2}\log (2n e^{-1}C_SD_0^2),
\end{align*}
which is 
\begin{align}
\frac{|B_{\bar{t}}(x,r)|}{r^n}\ \ge\ C_{VR}^-(T)(C_R,D_0,T)VD_0^{-n},
\end{align}
where $C_{VR}^-(T):=(2n e^{-1}C_S)^{-\frac{n}{2}}
\exp(2T(2C_R+D_0^{-2})-3^{n+5})$, which ultimately also depends on the
$L^2$-Poincar\'e constant $C_P$ and the doubling constant $C_D$ of the initial
metric, as encoded in $C_S$. Again, $C_{VR}^-(T)$ is invariant under the
parabolic rescaling of the Ricci flow.
\end{proof}

We now prove the local volume doubling property (\ref{eqn: local_doubling}),
which follows directly from the original contradiction argument due to
Grisha Perelman, if we notice that the constant $C_{VR}^-(T)$ is independent of
$\bar{t}$ and $r^2$ as long as $\bar{t}\le T$ and $r^2\le \bar{t}$.

Now suppose (\ref{eqn: VR_lb}) fails for some scale $r\in (0,\sqrt{\bar{t}})$ at
time $\bar{t}\le T$ and a point $x\in M$, then (\ref{eqn: local_doubling}) must
fail for this $r$, and it will also fail at scale $\frac{r}{2}$: otherwise, the
above argument applied to the $\frac{r}{2}$-ball around $x\in M$ will produce 
\begin{align*}
|B_{\bar{t}}(x,\frac{r}{2})|\ \ge\ &2^{-n}C_{VR}^-(T)(C_R,D_0,T)Vr^n\\
\ge\ &2^{-n}|B_{\bar{t}}(x,r)|,
\end{align*} 
where we have used the converse of (\ref{eqn: VR_lb}), but contradicts the
failure of (\ref{eqn: local_doubling}). Therefore, if the converse of
(\ref{eqn: local_doubling}) is observed at any point and scale, then it will
pass down to all smaller scales at that point, i.e. the converse of (\ref{eqn:
local_doubling}) implies for any $k\ge 1$,
\begin{align*}
|B_{\bar{t}}(x,2^{-k}r)|\ \le\ 3^{-nk}|B_{\bar{t}}(x,r)|,
\end{align*}
which is impossible for $k$ sufficiently large, since $(M,g(\bar{t}))$ is
locally Euclidean. Therefore, we have the following
\begin{proposition}[Lower bound of renormalized volume ratio]
Let $(M,g(t))$ be a Ricci flow solution on $[0,T]$ with initial diameter $D_0$
and initial volume $V$. Assume that the scalar curvature is uniformly bounded by
$C_R$ in space-time, then there is a constant $C_{VR}^-(T)$ depending on the
initial doubling constant $C_D$, the initial $L^2$-Poincar\'e constant $C_P$, the
initial diameter $D_0$, the scalar curvature bound $C_R$ and $T$, such that for any time
$t\in [0,T]$ and any scale $r$ such that $r^2\in (0,t]$,
\begin{align}
\frac{|B_{t}(x,r)|}{r^n}\ \ge\ C^+_{VR}(T) VD_0^{-n}.
\label{eqn: volume_ratio_lb}
\end{align}
Moreover, $C_{VR}^-(T)$ is invariant under the parabolic rescaling of the Ricci
flow.
\label{prop: volume_ratio_lb}
\end{proposition}

\subsection{Diameter upper bound}
\addtocontents{toc}{\protect\setcounter{tocdepth}{1}}
In the same vein, but with the straightforward estimate (\ref{eqn: estimate_Sc})
replaced by a more delicate maximal function argument, Peter
Topping~\cite{Topping05} proved a diameter upper bound in terms of the integral
of the scalar curvature (see also~\cite{Zhang14}). When the scalar curvature is
uniformly bounded in space-time, we notice that Topping's estimates depend on
the initial volume, a factor that we hope to avoid in our estimates.
However, once the quantities involved are correctly renormalized and the
initial entropy lower bound (\ref{eqn: initial_W_lb}) is used, Topping's
argument still leads to a diameter upper bound which is independent of the
initial volume. In the current subsection we discuss this in detail.

To begin with, we recall that the total volume changes as following:
\begin{align}\label{eqn: total_vol_1}
V(t)\ :=\ |M|_{g(t)}\ =\
Ve^{-\int_0^t\int_M\Sc_{g(s)}\dvol_{g(s)}\text{d}t}\ \le\ Ve^{C_Rt}.
\end{align}
Moreover, as discussed above, we evolve the cut-off function $h^2$
backward by the conjugate heat equation along the Ricci flow. From (\ref{eqn:
initial_W_lb}), (\ref{eqn: estimate_eA}), (\ref{eqn: estimate_gradient}) and
(\ref{eqn: estimate_ulogu}) we get
\begin{align}
\begin{split}
&\log VD_0^{-n} +C_1(T)\\
 \le\ &\frac{64|B_{\bar{t}}(x,r)|}{|B_{\bar{t}}(x,\frac{r}{2})|}+
 \frac{r^2}{|B_{\bar{t}}(x,\frac{r}{2})|}
 \int_{B_{\bar{t}}(x,r)}|\Sc_{g(\bar{t})}|\ \dvol_{g(\bar{t})}
 +\log\left(\frac{|B_{\bar{t}}(x,r)|}{r^n}\right),
\label{eqn:diameter_entropy_lower_bound}
\end{split}
\end{align}
where $C_1(T):=-2T(C_R+D_0^{-2})-\frac{n}{2}\log (2ne^{-1}C_S)$, a
constant only depending on the initial doubling and $L^2$-Poincar\'e constants,
the initial diameter and the space-time scalar curvature bound; especially it is
independent of the initial volume.

Now we define the maximal function of the scalar curvature
following~\cite{Topping05}:
\begin{align}
M\Sc(x,r,\bar{t}):=\sup_{s\in(0,r]}\frac{|B_{\bar{t}}(x,s)|}{s}
\left(\fint_{B_{\bar{t}}(x,s)}|\Sc_{g(\bar{t})}|\
\dvol_{g(\bar{t})}\right)^{\frac{n-1}{2}}.
\end{align}

We also define $C_2=\min\left\{\frac{\omega_nD_0^n}{2V},
e^{C_1(T)-2^{n+1}}\right\}$, where $\omega_n$ is the volume of
$n$-dimensional Euclidean unit ball. Notice that since we are dealing with the
case as $V\rightarrow 0$, the constant $C$ is in fact independent of $V$, and
we put it here just for the convenient of statement.

The key property of $M\Sc(x,r,\bar{t})$, as described in~\cite{Topping05}, is
that ``\emph{we cannot simultaneously have small curvature and small volume
ratio}'', but in our context we should consider the renormalized volume ratio
instead, and this is described in the following proposition:
\begin{lemma}
Let $(M,g(t))$ be a Ricci flow solution on $[0,T]$ with initial diameter $D_0$
and initial volume $V$. Assume that the scalar curvature is uniformly bounded by
$C_R$ in space-time, then for any $\bar{t}\in (0,T]$ and $r>0$ such that $r^2\le
\bar{t}$, we have
\begin{align*}
|B_{\bar{t}}(x,r)|\ \le\
C_2VD_0^{-n}r^n\ \Rightarrow\ M\Sc(x,r,\bar{t})\ \ge\ C_2VD_0^{-n},
\end{align*}
where the constant $C_2$ is defined as above.
\label{lem: Topping}
\end{lemma}
\begin{proof}[Proof (following~\cite{Topping05})]
We first claim that if $M\Sc(x,r,\bar{t})\le C_2VD_0^{-n}$ then for any $s\in
(0,r]$, $$|B_{\bar{t}}(x,s)|\ \le\ C_2VD_0^{-n}s^n\ \Rightarrow\
|B_{\bar{t}}(x,\frac{s}{2})|\ \le\ 2^{-n}C_2VD_0^{-n}s^n.$$

Suppose otherwise, then we could fix some $s\in (0,r]$ that contradicts the claim, i.e. 
\begin{align*}
|B_{\bar{t}}(x,\frac{s}{2})|\ >\ 
(C_2VD_0^{-n})^{\frac{2}{n-1}}2^{-n}s^{\frac{2n}{n-1}}|B_{\bar{t}}(x,s)|^{\frac{n-3}{n-1}},
\end{align*}
so that we have
\begin{align*}
\begin{split}
\int_{B_{\bar{t}}(x,s)}|\Sc_{g(\bar{t})}|\ \dvol_{g(\bar{t})}\
 \le\ &(M\Sc(x,r,\bar{t}))^{\frac{2}{n-1}}s^{\frac{2}{n-1}}
|B_{\bar{t}}(x,s)|^{\frac{n-3}{n-1}}\\
 \le\
 &(C_2VD_0^{-n})^{\frac{2}{n-1}}s^{\frac{2}{n-1}}|B_{\bar{t}}(x,s)|^{\frac{n-3}{n-1}}\\
<\ &2^n s^{-2}|B_{\bar{t}}(x,\frac{s}{2})|.
\end{split}
\end{align*}
By (\ref{eqn:diameter_entropy_lower_bound}), we could further deduce
\begin{align}
\begin{split}
\log VD_0^{-n}+C_1(T)\
\le\ &\frac{64|B_{\bar{t}}(x,s)|}{|B_{\bar{t}}(x,\frac{s}{2})|}
+\frac{s^2}{|B_{\bar{t}}(x,\frac{s}{2})|}\int_{B_{\bar{t}}(x,s)}|\Sc_{g(\bar{t})}|\
\dvol_{g(\bar{t})}+\log\left(\frac{|B_{\bar{t}}(x,s)|}{s^n}\right)\\
\le\ &\frac{64|B_{\bar{t}}(x,s)|}{|B_{\bar{t}}(x,\frac{s}{2})|}+2^n+\log
VD_0^{-n}+\log C_2,
\end{split}
\end{align}
so that $|B_{\bar{t}}(x,s)|\ge 2^n|B_{\bar{t}}(x,\frac{s}{2})|$ by the choice
of $C_2$, whence the claim.

Now by the claim, if there were any $x\in M$ that has some scale $s\in(0,r]$
contradicting the statement of the proposition, i.e.
$|B_{\bar{t}}(x,s)|\le C_2VD_0^{-n}s^n$ and simultaneously $M\Sc(x,s,\bar{t})\le
C_2VD_0^{-n}$, then for any $m\in \mathbb{N}$, we have, by the choice of $C_2$,
that
\begin{align*}
|B_{\bar{t}}(x,2^{-m}s)|\ \le\
2^{-mn}s^nC_2VD_0^{-n}\ \le\ \frac{\omega_n(2^{-m}s)^n}{2},
\end{align*}
which is impossible for \emph{all} $m$ sufficiently large, since as a smooth
Riemannian manifold, $(M,g(\bar{t}))$ is locally Euclidean of dimension $n$.
\end{proof}

Now we define $\nu:= \min\{(ne^{-1}C_S)^{\frac{n}{2}},1\}$. Notice that $\nu$
only depends on the initial data: the initial doubling and $L^2$-Poincar\'e
constants. We prove the following diameter bound:
\begin{proposition}
Let $(M,g(t))$ be a Ricci flow solution on $[0,T]$ with initial diameter $D_0$
and initial volume $V$, and assume that the scalar curvature is uniformly
bounded by $C_R$ in space-time. 

Then there is a constant $C_{diam}>0$ such that
if $VD_0^{-n}<\nu \omega_n$, then
\begin{align}
\forall t\in [0,T],\quad \diam (M,g(t))\le C_{diam}e^{2C_Rt}D_0,
\label{eqn: diam_ub}
\end{align}
where the constants only depend on $C_D,C_P,C_R,D_0$ and are invariant under the
parabolic rescaling of the Ricci flow.
\label{prop: diam_ub}
\end{proposition}
\begin{proof}[Proof (following~\cite{Topping05})]
By the assumption on $V$, we have, by its definition, $C_2=e^{C_1(T)-2^{n+1}}$.
For fixed $\bar{t}\in [0,T]$, let $\gamma$ be a minimal geodesic in $M$ with
$|\gamma|_{g(\bar{t})}= \diam(M,g(\bar{t}))$, and let $\{x_i\}$ be a maximal
set of points on $\gamma$ such that
\begin{enumerate}
  \item $B_{\bar{t}}(x_i,D_0\slash 10)$ are mutually disjoint; and
  \item $|B_{\bar{t}}(x_i,D_0\slash 10)|> 10^nC_2V$ for each $i$.
\end{enumerate}
Let $N:=|\{x_i\}|$, then clearly $N\le V(\bar{t})\slash (10^nC_2V)\le
e^{C_R\bar{t}}\slash C_2$.

Now the set $\gamma\backslash \cup_{i=1}^NB_{\bar{t}}(x_i,D_0\slash 5)$ has at
most $N+1$ connected components, and let $\sigma$ be one of these components
with largest length. We either have $\diam (M,g(\bar{t}))=|\gamma|_{g(\bar{t})}=
|\sigma|_{g(\bar{t})}$ if $N=0$; or else, if $N\ge 1$, we then have
\begin{align}
\begin{split}
\diam (M,g(\bar{t}))\ \le\ &(N+1)|\sigma|_{g(\bar{t})}+2ND_0\slash 5\\
\le\ &2N(|\sigma|_{g(\bar{t})}+D_0\slash 5)\\
\le\ &2C_2^{-1}e^{C_R\bar{t}}(|\sigma|_{g(\bar{t})}+D_0\slash 5).
\end{split}
\label{eqn: diameter_1}
\end{align}
In any case, we will need to estimate $|\sigma|_{g(\bar{t})}$ in terms of the
initial diameter $D_0$:

For any $x\in Im(\sigma)$, the maximality of $\{x_i\}$ guarantees that
$|B_{\bar{t}}(x,D_0\slash 10)|\le 10^nC_2V$. Now by Lemma~\ref{lem: Topping}, we
know that
\begin{align*}
\forall x\in Im(\sigma),\quad M\Sc(x,D_0\slash 10,\bar{t})\ \ge\ C_2VD_0^{-n}.
\end{align*}
Therefore we could find some $s(x)\in (0,D_0\slash 10]$ such that 
\begin{align*}
C_2VD_0^{-n}\ \le\ 
&\frac{|B_{\bar{t}}(x,s)|}{s}\left(\fint_{B_{\bar{t}}(x,s)}|\Sc_{g(\bar{t})}|\
\dvol_{g(\bar{t})}\right)^{\frac{n-1}{2}}\\
\le\ &\frac{1}{s}
\int_{B_{\bar{t}}(x,s)}|\Sc_{g(\bar{t})}|^{\frac{n-1}{2}}\ \dvol_{g(\bar{t})}.
\end{align*}
Now we could apply Lemma 5.2 of~\cite{Topping05} to pick a set of points
$\{y_j\}\subset Im(\sigma)$ such that $\{B_{\bar{t}}(y_j,s(y_j))\}$ are mutually
disjoint, and that $|\sigma|\le 6\sum_j s(y_j)$. We could now estimate
\begin{align}
\begin{split}
|\sigma|_{g(\bar{t})}\ \le\ &6\sum_j\frac{D_0^n}{C_2V}
\int_{B_{\bar{t}}(y_j,s(y_j))}|\Sc_{g(\bar{t})}|^{\frac{n-1}{2}}\
\dvol_{g(\bar{t})}\\
\le\ 
&\frac{6D_0^n}{C_2V}\int_M|\Sc_{g(\bar{t})}|^{\frac{n-1}{2}}\
\dvol_{g(\bar{t})}\\
\le\ &6C_2^{-1}e^{C_R\bar{t}}D_0^{n}C_R^{\frac{n-1}{2}}.
\end{split}
\label{eqn: diameter_2}
\end{align}
Putting the estimates (\ref{eqn: diameter_1}) and (\ref{eqn:
diameter_2}) together we obtain
\begin{align*}
\diam(M,g(\bar{t}))\ \le\
C_{diam}e^{2C_R\bar{t}}D_0,
\end{align*}
where, recalling the definition of $C_2$, we have
$C_{diam}=4^{n+4}D_0^{n-1}C_R^{\frac{n-1}{2}}.$ 

We notice that $C_{diam}$ is invariant under the parabolic rescaling of the
Ricci flow, and is independent of time.
\end{proof}

\begin{remark}
When the scalar curvature is uniformly bounded, the previous renormalized volume
ratio lower bound, and the upper bound of the total volume, actually provide a
diameter upper bound. This naive estimate, however, fails to provide constants
that are invariant under the parabolic rescaling of the Ricci flow.
\end{remark}

\subsection{A weak compactness result}
We now state a proposition that corresponds to our second motivation of the
paper: to constuct, from a sequence of Ricci flows with collapsing initial
data, Gromov-Hausdorff limits of the positive time-slices. Compare a result
of Chen-Yuan~\cite[Theorem 1]{CY} in the case where lower bounds, uniform in
space-time, of the Ricci curvature and the unit ball volume are assumed.

\begin{prop}[Weak compactness for positive time slices]\label{prop:
weak_compactness} 
Let $\{(M_i,g_i(t))\}$ be a sequence of Ricci flows defined for $t\in [0,T]$,
such that they satisfy the same assumptions as in Theorem~\ref{thm: main}.

Then for each $t\in (0,T)$, there is a subsequence of $\{(M_i,g_i(t))\}$, a
compact metric space $(X_t,d_t)$, to which the subsequence converges in the
Gromov-Hausdorff topology.
\end{prop}
\begin{proof}
This is a simple consequences of the estimates we proved previously in this
section. Recall that in \cite[Chapter 5, A]{Gromov_Metric} a quantity
$N(\varepsilon, R,X)$ is defined for each complete metric space $X$, to denote
the maximal number of disjoint $\varepsilon$-balls that could be possibly fitted
into an $R$-ball in the metric space $X$. As shown in \cite[Proposition
5.2]{Gromov_Metric}, as long as $N(\varepsilon, R,X_i)$ is uniformly bound for
all $\varepsilon\in (0,R)$, $R\in (0,\diam X_i)$ and $X_i$, the sequence
$\{X_i\}$ is precompact in the pointed-Gromov-Hausdorff topology. 

In our situation, $\forall t\in (0,T)$, since we have a uniform diameter upper
bound (\ref{eqn: diam_ub}), we only need to control
$N\left(\varepsilon,C_{diam}e^{2C_Rt}D_0,(M_i,g_i(t))\right)$. In fact, we could
easily see that $\forall \varepsilon\in (0,C_{diam}e^{2C_Rt}D_0)$, the total
volume upper bound (\ref{eqn: total_vol_1}) together with the lower bound of
renormalized volume ratio (\ref{eqn: volume_ratio_lb}) gives: denoting
$V_i:=Vol(M_i,g_i(0))$, we have
\begin{align*}
N\left(\varepsilon,C_{diam}e^{2C_Rt}D_0,(M_i,g_i(t))\right)\ \le\
\frac{V_ie^{C_Rt}}{C_{VR}^+(T)V_iD_0^{-n}\varepsilon^n}\ =\
\frac{e^{C_Rt}}{C_{VR}^+}\left(\frac{D_0}{\varepsilon} \right)^n.
\end{align*}
This bound is uniform on the sequence $\{(M_i,g_i(t))\}$ and therefore there is
a metric space $(X_t,d_t)$ to which the sequence subconverges in the
Gromov-Hausdorff sense. 

Clearly, $\diam (X_t,d_t)\le C_{\diam}e^{2C_Rt}$.
\end{proof}

\begin{remark}
Especially, we may assume $\lim_{i\to \infty}V_iD_0^{-n}= 0$, and this
justifies us calling ``collapsing initial data''.
\end{remark}

\section{Estimating the analytic quantities along the Ricci flow}
In this section we prvide a rough upper bound of the renormalized heat kernel,
following Davies' argument as discussed in Qi S. Zhang's book and paper; we then
apply this rough upper bound, the Harnack inequality (\ref{eqn: Harnack}), and
the diameter upper bound (\ref{eqn: diam_ub}) to obtain an on-diagonal lower
bound lower bound of the renormalized heat kernel, when the initial global
volume ratio is sufficiently small. As consequences, we also deduce a Gaussian
type lower bound of the renormalized heat kernel, as well as the non-inflation
property for the renormalized volume ratio.

\subsection{Rough upper bound of the renormalized heat kernel in space time}
\addtocontents{toc}{\protect\setcounter{tocdepth}{1}}

The technique used to show the uniform Sobolev inquality in Subsection 3.2,
i.e. the method of Davies, could be further applied to obatin a rough
upper bound of the heat kernel coupled with the Ricci flow. This was first
noticed by Qi S. Zhang in~\cite{Zhang12} and we will follow the exposition
there. Notice that recently, Meng Zhu also extended Davies' method to Ricci
flows with a uniform Ricci curvature lower bound in space-time,
see~\cite{Zhu16}.

We fix $(x_0,t_0)\in M\times [0,T)$, and consider the heat kernel based at
$(x_0,t_0)$, coupled with the Ricci flow, as introduced in Subsection 2.3.
More specifically, we denote the heat kernel by 
\begin{align*}
K(x,t)=G(x_0,t_0;x,t),
\end{align*}
i.e. for any $(x,t)\in M\times (t_0,T]$,
$(\partial_t-\Delta_{g(t)})K(x,t)\ =\ 0$, 
and $\lim_{t\downarrow t_0}K(x,t)\ =\ \delta_{x_0}(x)$.

Now fix any $t\in (t_0,T]$, let $p(s):=(t-t_0)\slash (t-s)$ for $s\in (t_0,t]$.
Besides $p(t_0)=1$ and $\lim_{s\uparrow t}p(s)=\infty$, we also notice the
following relations:
\begin{align*}
0\le \frac{p(s)-1}{p'(s)}\ &=\ \frac{(s-t_0)(t-s)}{t-t_0}\ \le\ 
t-t_0,\\
 0<\frac{1}{p'(s)}\ &=\ \frac{(t-t_0)^2}{t}\ \le\ t,\\
\text{and}\quad p'(s)p^{-2}(s)\ &=\ \frac{1}{t-t_0}.
\end{align*}
Defining for any $(x,s)\in M\times (t_0,t]$,
\begin{align*}
v(x,s)\ :=\ K(x,s)^{\frac{p(s)}{2}}\| K^{\frac{p(s)}{2}}\|^{-1}_{L^2(M,g(s))},
\end{align*}
we could compute as before to obtain
\begin{align*}
\partial_s\log \| K\|_{L^{p(s)}(M,g(s))}\ =\ &\frac{p'(s)}{p^2(s)}\int_Mv^2\log
v^2-\frac{4(p(s)-1)}{p'(s)}|\nabla v|^2-\frac{p(s)}{p'(s)}\Sc_{g(s)}v^2\
\dvol_{g(s)}\\
\le\ &\frac{p'(s)}{p^2(s)}\int_Mv^2\log
v^2-\frac{p(s)-1}{p'(s)}\left(4|\nabla
v|^2+\Sc_{g(s)}v^2\right)\ \dvol_{g(s)}+C_R.
\end{align*}

We now plug $\tau=\frac{p(s)-1}{p'(s)}$ and $\bar{t}=s$ into the uniform
log-Sobolev inequality (\ref{eqn: uniform_log_Sobolev}) (a bit abusing of
notation), and obtain from the above calculations:
\begin{align*}
\begin{split}
&\partial_s\log \| K\|_{L^{p(s)}(M,g(s))}\\
 \le\ &\frac{p'(s)}{p^2(s)}\left(-\frac{n}{2} \log
 \frac{p(s)-1}{p'(s)}-\log VD_0^{-n}+(C_R+D_0^{-2})
 \left(\bar{t}+\frac{p(s)-1}{p'(s)}\right)+C_{lS}\right)+C_R.
\end{split}
\end{align*}
Integrating $s$ from $t_0$ to $t$, we see that
\begin{align*}
\log\frac{\|
K(-,t)\|_{L^{\infty}(M,g(t))}}{\| K(-,t_0)\|_{L^1(M,g(t_0))}}\
\le\ -\log
VD_0^{-n}(t-t_0)^{\frac{n}{2}}+2t(C_R+D_0^{-2})+C_R(t-t_0)+C_{lS}+n.
\end{align*}
Since the $K(-,s)$ acquires the Delta function property as $s$ descends to
$t_0$, we clearly have 
\begin{align*}
\| K(-,t_0)\|_{L^1(M,g(t_0))}=1.
\end{align*} 
Therefore, exponentiating both sides of the above estimate we obtain
\begin{align*}
VD_0^{-n}G(x_0,t_0;x,t)\ \le\ C_H^+(T)(t-t_0)^{-\frac{n}{2}},
\end{align*}
where $C_H^+(T)=\exp(2T(C_R+D_0^{-2})+C_RT+C_{lS}+n)$ is a universal
constant independent of $t-t_0$, and is invariant under the parabolic
rescaling of the Ricci flow. We collect the result in the following
\begin{proposition}[Rough upper bound of the renormalized heat kernel]
Let $(M,g(t))$ be a Ricci flow solution on $[0,T]$ with initial diameter $D_0$
and initial volume $V$, and assume that the scalar curvature is uniformly
bounded by $C_R$ in space-time. 

Then there is a constant $C_H^+(T)=C_H^+(C_D,C_P,C_R,D_0,n,T)$ such that for
any $(y,t)\in M\times (0,T]$, the conjugate heat kernel $G(-,-;y,t)$ based at
$(y,t)$ obeys the following estimate:
\begin{align}
\forall x\in M,\ \forall s\in (0,t),\quad 
VD_0^{-n}G(x,s;y,t)\ \le\ C_H^+(T)(t-s)^{-\frac{n}{2}}.
\label{eqn: heat_ub}
\end{align}
\label{prop: heat_ub}
\end{proposition}

\subsection{Gaussian type lower bound of the renormalized heat kernel}
\addtocontents{toc}{\protect\setcounter{tocdepth}{1}}
In~\cite{CW13} and~\cite{Zhang12}, the non-inflation property of volume
ratio was proven based on an on-diagonal heat kernel lower bound, which was
obtained by estimating the reduced length of a constant curve in space, and
Perelman's estimate. Such on-diagonal heat kernel lower bound, together with
the gradient estimate Theorem~\ref{thm: gradient_estimate}, then gives a
Gaussian lower bound of the heat kernel. This lower bound is essential in
Bamler-Zhang's estimate of the distance distortion.

We notice that this lower bound, however, could be not applied in the case of
collapsing initial data, basically because of the lack of a correct
renormalization. In this subsection, our major task is then to obtain a Gaussian
type lower bound of the renormalized heat kernel. We will start similarly with
an on-diagonal lower bound, and then apply the gradient estimate to obtain the
desired Gaussian lower bound of the renormalized heat kernel.

First let us recall the volume upper bound:
\begin{align}
V(t):=|M|_{g(t)}=Ve^{-\int_0^t\int_M\Sc(s)\dvol_{g(s)}\text{d}t}\le Ve^{C_Rt}.
\end{align}
We also recall the bound of the total heat on the whole manifold: for any
$x,y\in M$ and any $0\le s<t\le T$, recall that $G(x,s;y,t)$ satisfies the heat
equation in the variable $(y,t)$, and it satisfies the conjugate heat equation
in the variable $(x,s)$; fixing $(x,s)$ and integrating $y\in M$ we have
$\forall t'<t\le T$,
\begin{align*}
\left|\partial_t\int_MG(x,t';y,t)\ \dvol_{g(t)}\right|\ =\
&\left|\int_M\Delta_yG(x,t';y,t)-\Sc_{g(t)}(y)G(x,t';y,t)\
\dvol_{g(t)}(y)\right|\\
\le\ &C_R\int_MG(x,t';y,t)\ \dvol_{g(t)}(y),
\end{align*}
therefore integrating in time we see
\begin{align}
e^{-C_R(t-s)}\ \le\ \int_MG(x,s;y,t)\ \dvol_{g(t)}(y)\ \le\ e^{C_R(t-s)}.
\label{eqn: heat_bound}
\end{align}

As discussed in the introduction, the collapsing procedure is an intrinsic
geometric phenomenon, and it should not cause the addition or loss of the total
heat. Therefore, the heat-volume duality should be preserved. If the heat kernel
at a future time fails to have a pointwise lower bound of order $(VD_0^{-n})^{-1}$,
then the duality between the heat and volume will be contradicted, due to
the rough bound of the renormalized heat kernel, the Harnack inequality
(\ref{eqn: Harnack}), and the volume and diameter upper bounds of the whole
manifold.

More specifically, the diameter upper bound in Proposition~\ref{prop:
diam_ub}, the rough pointwise upper bound of $G(x,s;-,-)$ in
Propostition~\ref{prop: heat_ub}, joint with the Harnack inequality (\ref{eqn:
Harnack}) give, for any $(y,t)\in M\times (s,T]$,
\begin{align*}
 G(x,s;y,t)\ \le\
H(n)\left(\frac{C_H^+(T)G(x,s;x,t)}{(VD_0^{-n})(t-s)^{\frac{n}{2}}}\right)^{\frac{1}{2}}
\exp\left(\frac{H'(n)C^2_{diam}e^{4C_RT}D_0^2}{(t-s)}\right),
\end{align*}
whenever the initial golbal volume ratio is bounded as $VD_0^{-n}\le \nu
\omega_n$.

Now we let $\theta:=\sqrt{t-s}\slash D_0$ denote the ratio of the parabolic
scale to the initial diameter. Integrating $y\in M$ and involking the
lower bound of total heat in (\ref{eqn: heat_bound}), we see
\begin{align*}
e^{-C_R(t-s)}\ &\le\ \int_MG(x,s;y,t)\ \dvol_{g(t)}(y)\\
&\le\
H(n)Ve^{C_Rt}\left(\frac{C_H^+(T)G(x,s;x,t)}{V\theta^n}\right)^{\frac{1}{2}}
\exp\left(\frac{H'(n)C^2_{diam}e^{4C_RT}}{\theta^2}\right)\\
&=\
H(n)e^{C_Rt}\theta^{-\frac{n}{2}}\left(C_H^+(T)VG(x,s;x,t)\right)^{\frac{1}{2}}
\exp\left(\frac{H'(n)C^2_{diam}e^{4C_RT}}{\theta^2}\right),
\end{align*}
and thus
\begin{align*}
VG(x,s;x,t)\ \ge\
H(n)^{-2}e^{-C_R(3t-s)}(C_H^+(T))^{-1}\theta^n
\exp\left(-\frac{2H'(n)C^2_{diam}e^{4C_RT}}{\theta^2}\right).
\end{align*}
Multiplying $D_0^{-n}$ on both sides of this inequality we get
\begin{align*}
VD_0^{-n}G(x,s;x,t)\ \ge\ C_{HD}^-(T)\Psi(\theta\ |\ T)(t-s)^{-\frac{n}{2}},
\end{align*}
where 
\begin{align*}
C_{HD}^-(T)\ :=\ H(n)^{-2}e^{-3C_RT}C_H^+(T)^{-1},\quad
\text{and}\quad
\Psi(\theta\ |\ T)\
:=\ \theta^{2n}\exp\left(-2H'(n)C^2_{diam}e^{4C_RT}\theta^{-2}\right).
\end{align*}
Especially, we notice that $C_{HD}^-(T)$ only depends on the initial diameter,
the doubling and $L^2$-Poincar\'e constants, the space-time scalar curvatur
upper bound, and the time elapsed from the beginning. On the other hand, we notice that
$\Psi(\theta\ |\ T)$ depends, besides $\theta$ and $T$, only on the initial
diameter and the space-time scalar curvature bound, especially it is \emph{independent}
of the initial doulbing and $L^2$-Poincar\'e constants. We clearly see that 
\begin{align}
\lim_{\theta\to 0}\Psi(\theta\ |\ T)\ =\ 0,
\end{align}
indicating that the renormalization is only valid on scales comparable to the
initial diameter. Moreover, both constants $C_{HD}^-(T)$ and $\Psi(\theta\ |\
T)$ are invariant under the parabolic rescaling of the Ricci flow. Summarizing, we
have the following
 \begin{lemma}[On-diagonal lower bound of the renormalized heat kernel] Let
 $(M,g(t))$ be a Ricci flow solution on $[0,T]$ with initial diameter $D_0$ and
 initial volume $V$, and assume that the scalar curvature is uniformly bounded
 by $C_R$ in space-time. 
 
 Then there are positive constants 
 \begin{align*}
 C_{HD}^-(T)=C_{HD}^-(C_D,C_P,C_R,D_0,n,T)\quad \text{and}\quad
 \Psi(\theta\ |\ T)=\Psi(\theta\ |\ C_R,D_0,n,T)
  \end{align*} 
  such that for any $(x,t)\in M\times (0,T]$, the conjugate heat kernel
  $G(-,-;x,t)$ based at $(x,t)$ obeys the following estimate: for any $s\in
  (0,t)$, setting $\theta:=\sqrt{t-s}\slash D_0$, then
\begin{align}
 VD_0^{-n}G(x,s;x,t)\ \ge\ C_{HD}^-(T)\Psi(\theta\ |\ T)(t-s)^{-\frac{n}{2}},
\label{eqn: heat_diagonal}
\end{align}
whenever $VD_0^{-n}\le \nu \omega_n$. Moreover, the constants $C_{HD}^-(T)$ and
$\Psi(\theta\ |\ T)$ are invariant under the parabolic rescaling of the Ricci
flow, and $\lim_{\theta\to 0}\Psi(\theta\ |\ T)=0$.
\label{lem: heat_diagonal}
\end{lemma}
\noindent Recall that the positive constant $\nu$ is defined right above
Proposition~\ref{prop: diam_ub}.

Once this on-diagonal estimate is obtained, we could easily apply the Harnack
inequality (\ref{eqn: Harnack}) again to obtain a Gaussian lower bound:
\begin{proposition}[Gaussian type lower bound of the renormalized heat kernel]
Let $(M,g(t))$ be a Ricci flow solution on $[0,T]$ with initial diameter $D_0$
and initial volume $V$, and assume that the scalar curvature is uniformly
bounded by $C_R$ in space-time. 

Then there are positive constants 
\begin{align*}
C_{H}^-(T)=C_H^-(C_D,C_P,C_R,D_0,n,T)\quad \text{and}\quad
\Psi(\theta\ |\ T)=\Psi(\theta\ |\ C_R,D_0,n,T)
\end{align*} 
such that for any $(y,t)\in M\times (0,T]$, the conjugate heat kernel
$G(-,-;y,t)$ based at $(y,t)$ obeys the following estimate: for any $s\in
(0,t)$, setting $\theta:=\sqrt{t-s}\slash D_0$, then
\begin{align}
 VD_0^{-n}G(x,s;y,t)\ \ge\
 C_{H}^-(T)\Psi(\theta\ |\ T)^2(t-s)^{-\frac{n}{2}}
 \exp\left(-2H'(n)\frac{d_t(x,y)^2}{t-s}\right),
 \label{eqn: heat_Gaussian_lb}
\end{align}
whenever $VD_0^{-n}\le \nu \omega_n$. Moreover, the constants $C_{H}^-(T)$ and
$\Psi(\theta\ |\ T)$ are invariant under the parabolic rescaling of the Ricci
flow, and $\lim_{\theta\to 0}\Psi(\theta\ |\ T)=0$.
\label{prop: heat_Gaussian_lb}
\end{proposition}
\noindent Here the constant is defined as
$C_H^-(T):=(C_{HD}^-(T))^2(H(n)^2C_H^+(T))^{-1}$.

As a direct geometric consequence, we could also deduce the non-inflation
property of the volume ratio:
\begin{corollary}[Non-inflation of the renormalized volume ratio]
Let $(M,g(t))$ be a Ricci flow solution on $[0,T]$ with initial diameter $D_0$
and initial volume $V$, and assume that the scalar curvature is uniformly
bounded by $C_R$ in space-time. 

Then there are positive constants 
\begin{align*}
C_{VR}^+(T)=C_{VR}^+(C_D,C_P,C_R,D_0,n,T)\quad
\text{and}\quad \Psi(\theta\ |\ T)=\Psi(\theta\ |\ C_R,D_0,n,T)
\end{align*}
such that for any $(x,t)\in M\times (0,T]$ and any $r\in (0,\sqrt{t})$, setting
$\theta=r\slash D_0$, then
\begin{align*}
(VD_0^{-n})^{-1}|B_t(x,r)|\ \le\ \frac{C_{VR}^+(T)}{\Psi(\theta\ |\ T)^2}r^n,
\end{align*}
whenever $VD_0^{-n}\le \nu \omega_n$. Moreover, the constants $C_{VR}^+(T)$ and
$\Psi(\theta\ |\ T)$ are invariant under the parabolic rescaling of the Ricci
flow, and $\lim_{\theta\to 0}\Psi(\theta\ |\ T)=0$.
\label{cor: non-inflation}
\end{corollary}

\begin{proof}
Fix $(x,t)\in M\times (0,T]$ and $r\in (0,\sqrt{t})$. Let $G(x,t-r^2;-,-)$ be
the fundamental solution to the conjugate heat equation coupled with the Ricci
flow on $M\times (t-r^2,T]$, based at $(x,t-r^2)$, i.e. $\lim_{s\downarrow
t-r^2}G(x,t-r^2;-,s)=\delta_{(x,t-r^2)}(-)$. By the Gaussian type lower bound
(\ref{eqn: heat_Gaussian_lb}) of the renormalized heat kernel, we have
\begin{align*}
\forall y\in B_t(x,r),\quad VD_0^{-n}G(x,t-r^2;y,t)\ \ge\
C_H^-(T)\Psi(\theta\ |\ T)^2e^{-2H'(n)}r^{-n}.
\end{align*}
On the other hand, by (\ref{eqn: heat_bound}), we have an upper bound of the
total heat. Therefore, integrating over $B_t(x,r)$ we have
\begin{align*}
e^{C_Rr^2}\ &\ge\ \int_{B_t(x,r)}G(x,t-r^2;y,t)\ \dvol_{g(t)}(y)\\
&\ge\ C_H^-(T)\Psi(\theta\ |\ T)^2e^{-2H'(n)} |B_t(x,r)|(VD_0^{-n}r^n)^{-1},
\end{align*}
or equivalently, 
\begin{align*}
(VD_0^{-n})^{-1}|B_t(x,r)|\ \le\
C_{VR}^+(T)\Psi(\theta\ |\ T)^{-2} r^{n},
\end{align*}
with $C_{VR}^+(T):=e^{2H'(n)+C_RT}\slash C_H^-(T)$.
\end{proof}
\begin{remark}
Again, we see that the bound becomes worse as $r\slash D_0$ becomes smaller.
However, for any fixed positive scale, we have a uniform estimate.
\end{remark}

\section{Estimating the distance distortion}
In this section we prove the main result of our note: the distance
distortion estimate. Once the lower bound of the renormalized volume ratio
(\ref{eqn: VR_lb}) matches with that of the renormalized heat kernel (\ref{eqn:
heat_diagonal}), the classical argument of counting geodesic balls suitably
covering a minimal geodesic carries over; see Section 5.3 of~\cite{CW14} and
Section 3 of~\cite{BZ15a}. See also Section 3 of Chen-Wang~\cite{CW16} for a
thorough exposition. However, we reproduce a detailed proof here,
following~\cite{BZ15a}, for the sake of completeness and readers' convenience.

Before the commencement of the proof, we would like to emphasize the importance
of the parabolic-scaling invariance of the constants in our previous estimates:
for fixed small scales, we will dialate them to unit size and work with the
rescaled quantities.

\begin{proof}[Proof of Theorem~\ref{thm: main}] Fix
$t_1\in (0,T]$, and suppose that $d_{t_1}(x_0,y_0)=r$. Let $\theta=r\slash D_0$.
Then we rescale $r$ to $1$ parabolically and denote the rescaled time slice as
$\bar{t}$. Also denote the rescaled metric as $\bar{g}$. 

Let $\gamma:[0,1]\to M$ be a unit speed $\bar{g}(\bar{t})$-minimal geodesic that
connects $x_0$ to $y_0$. Let 
\begin{align*}
K(x,t):=G(x_0,\bar{t}-\frac{1}{2};x,t)
\end{align*} 
be a heat kernel coupled with the Ricci flow, with initial data the Delta
function at $(x_0,\bar{t}-\frac{1}{2})$, 
recall immediately that we have the bound (\ref{eqn: heat_bound}) of the total
heat:
\begin{align}
\forall t\in [\bar{t}-\frac{1}{2},\bar{t}+\frac{1}{2}],\quad \int_{M}K(-,t)\
\dvol_{\bar{g}(t)}\ \le\ e^{C_Rr^{2}}.
\label{eqn: total_heat_ub}
\end{align}

By the renormalized heat kernel upper bound (\ref{eqn: heat_ub}), we have
\begin{align}
\forall t\in [\bar{t}-\frac{1}{4},\bar{t}+\frac{1}{4}],\quad (VD_0^{-n})K(-,t)\
\le\ C_H^+ 2^n;
\label{eqn: K_ub}
\end{align}
on the other hand, by the Gaussian type lower bound (\ref{eqn:
heat_Gaussian_lb}), we have
\begin{align}
\forall s\in [0,1],\quad (VD_0^{-n})K(\gamma(s),\bar{t})\ \ge\
C_H^-e^{-4H'(n)}2^{\frac{n}{2}}\Psi(\theta\ |\ T)^2.
\label{eqn: K_Gaussian}
\end{align}
Time derivative bound (\ref{eqn: time_derivative}) together
with (\ref{eqn: K_ub}) imply that
\begin{align*}
\forall (s,t)\in [0,1]\times [\bar{t}-\frac{1}{4},\bar{t}+\frac{1}{4}],
\quad |\partial_t(VD_0^{-n})K(\gamma(s),t)|\ \le\ C_H^+2^n(C_Rr^2+4B(n));
\end{align*}
therefore, setting
\begin{align*}
\alpha_0(\theta)\ :=\ \min\left\{\frac{1}{8},
\frac{C_H^-e^{-4H'(n)}\Psi(\theta\ |\ T)^2}{2^{n+1}C_H^+(C_RT+4B(n))}\right\}
\end{align*}
and integrating the above time derivative bound we obtain from (\ref{eqn:
K_Gaussian}) that
\begin{align*}
\forall (s,t)\in [0,1]\times
[\bar{t}-\alpha_0(\theta),\bar{t}+\alpha_0(\theta)], \quad
(VD_0^{-n})K(\gamma(s),t)\ \ge\ \frac{1}{2}C_H^-e^{-4H'(n)}
2^{\frac{n}{2}}\Psi(\theta\ |\ T)^2.
\end{align*}
Now by the Harnack inequality (\ref{eqn: Harnack}), we could estimate
\begin{align}
\forall (s,t)\in [0,1]\times [\bar{t}-\alpha_0,\bar{t}+\alpha_0],\quad
\inf_{B_t(\gamma(s),1)}(VD_0^{-n})K(-,t)\ \ge\
C_3(T)\Psi(\theta\ |\ T)^4,
\label{eqn: local_heat_lb}
\end{align}
where $C_3(T):=C_H^-(T)e^{-16H'(n)}\slash (4H(n)^2C_H^+(T))$ is a constant only
depending on the initial diameter, the initial doubling and $L^2$-Poincar\'e
constants, and the space-time scalar curvature bound. Moreover, $C_3(T)$ is
invariant under the parabolic rescaling of the Ricci flow.

Now fix any $t\in
[\bar{t}-\alpha_0(\theta),\bar{t}+\alpha_0(\theta)]$, and cover
$Im(\gamma)\subset M$ by a minimal number of unit $\bar{g}(t)$-geodesic balls 
$\{B_t(\gamma(s_i),1)\}$, $i=1,\cdots,N$. It is easily seen that
$|\gamma|_{\bar{g}(t)}\le 2 N$. Therefore, in order to obtain an upper
bound of $d_{t}(x_0,y_0)$, it suffices to control $N$ from above.

By the minimality of the covering, we see that the collection
$\{B_t(\gamma(s_i),1\slash 2)\}$ are pairwise disjoint. We could
therefore combine the upper bound (\ref{eqn: total_heat_ub}) of the total
heat, the renormalized lower bound (\ref{eqn: local_heat_lb}) of the local heat,
together with the lower bound (\ref{eqn: volume_ratio_lb}) of the renormalized
volume ratio, to estimate:
\begin{align*}
e^{C_Rr^2}\ \ge\ &\int_MK(x,t)\ \dvol_{\bar{g}(t)}\\
\ge\ &\sum_{i=1}^N\int_{B_t(\gamma(s_i),\rho\slash 2)}K(x,t)\
\dvol_{\bar{g}(t)}\\
\ge\ &N C_3(T)2^{-n}\Psi(\theta\ |\ T)^4.
\end{align*}
Therefore $N\le 2^ne^{C_RT}C_3(T)^{-1}\Psi(\theta\ |\ T)^{-4}$, a constant
independent of specific Ricci flow, especially its initial entropy. On the other
hand, recalling that $d_{\bar{t}}(x_0,y_0)=1$, we get
\begin{align}
\forall t\in [\bar{t}-\alpha_1(\theta),\bar{t}+\alpha_1(\theta)],\quad
d_{t}(x_0,y_0)\ \le\ \alpha_1(\theta)^{-1} d_{\bar{t}}(x_0,y_0),
\label{eqn: one_side}
\end{align}
where
\begin{align*}
\alpha_1(\theta):=\min\left\{\alpha_0(\theta),
\frac{C_3(T)\Psi(\theta\ |\ T)^4}{2^{n+1}e^{C_RT}}\right\}.
\end{align*}

This proves one side of the desired distance distortion estimate. To see the
other side, we notice that the estimate (\ref{eqn: one_side}) is independent of
specific time slice $\bar{t}$. Therefore, letting
$\alpha(\theta)=\frac{1}{2}\alpha_1(\theta)$, and applying the previous argument
at the $t$-slice for any $t\in[\bar{t}-\alpha,\bar{t}+\alpha]$, we see
\begin{align*}
\forall s\in [t-\alpha_1(\theta),t+\alpha_1(\theta)],\quad d_{s}(x_0,y_0)\le
\alpha_1(\theta)^{-1}d_{t}(x_0,y_0).
\end{align*} 
Especially, since $\bar{t}\in [t-\alpha_1,t+\alpha_1]$, plugging
$s=\bar{t}$ into the the above inequality we get the desired estimate
(\ref{eqn: main}) with $\alpha(\theta)$ in place of $\alpha_1(\theta)$.
\end{proof}

 Here we emphasize again that $\alpha(\theta)\to 0$ as $\theta\to 0$, reflecting
 the fact that when we look at smaller scales compared to the initial diameter,
 the estimate will be less effective.
  
We could also enhance the above distance distortion estimate in the following
\begin{corollary}
Let $(M,g(t))$ be a complete Ricci flow solution on $[0,T]$ with initial
diameter $D_0$ and initial volume $V$, and assume the following conditions:
\begin{enumerate}
 \item $(M, g(0))$, as a closed Riemannian manifold, has its doubling constant
 uniformly bounded above by $C_D$, and its $L^2$-Poincar\'e constant by $C_P$,
 and
  \item the scalar curvature is uniformly bounded in space-time: $\sup_{M\times 
  [0,T]}|\Sc_{g(t)}|\le C_R$.
\end{enumerate} 
There exist two positive constants
$\alpha=\alpha(\theta\ |\ C_D,C_P,C_R,D_0,n,T)<1$ with
\begin{align*}
\lim_{\theta\to 0}\ \alpha(\theta\ |\ C_D,C_P,C_R,D_0,n,T)= 0,
\end{align*}
 and $\nu=\nu(C_D,C_P,C_R,n)<1$, such that whenever $VD_0^{-n}\le \nu
 \omega_n$, we have,
 \begin{align*}
 &\forall t\in [0,T],\quad \forall x,y\in M\ \text{with}\ d_t(x,y)=:r\le
 \sqrt{t}\\
 \text{and}\quad &\forall s\in [r^2,T]\ \text{with}\ |s-t| \le
 \alpha(\theta)\min\{C_R^{-1},t\}+r^2,
 \end{align*} 
 the following estimate:
 \begin{align*}
 d_s(x,y)^2\ \le\ \alpha(\theta)^{-1}(d_t(x,y)^2+|s-t|).
 \end{align*}
\end{corollary}

\noindent The proof is identical to that of Corollary 1.2 of~\cite{BZ15a}, and
we will omit it here.

\section{Discussion}
\addtocontents{toc}{\protect\setcounter{tocdepth}{1}}
Although we have achieved the goal of our note --- pointing our that meaningful
geometric consequences could be proven from a correct renormalization, even if
a uniform initial $\mu$-entropy lower bound may be violated --- some new
research directions are left open, which we now briefly discuss.


\subsection{Uniform H\"older continuity of distance function}
In a sequential work~\cite{BZ15b}, Bamler-Zhang obtained a
$1\slash 2$-H\"older continuity of the distance function along the Ricci
flows, where the constant is uniformly bounded in terms of the scalar curvature
and the initial $\mu$-entropy. This result is parallel to Colding-Naber's
H\"older continuity theorem for manifolds with a uniform Ricci curvature lower
bound~\cite{CoNa11}. See the original work of Perelman~\cite{Perelman02} for
a comparison geometry viewpoint of the Ricci flow, as well as the recent work
of Bing Wang~\cite{Bing17} for a nice explanation of the similarities between
Ricci flows and manifolds with Ricci curvature lower bound.

 Notice however, that Colding-Naber's estimates are
independent of the uniform volume non-collapsing assumption, a condition
comparable to the uniform initial $\mu$-entropy lower bound in the Ricci flow
setting. It is therefore natural to conjecture that Bamler-Zhang's H\"older
continuity could be generalized to Ricci flows with uniformly bounded scalar
curvature in space-time, but \emph{without} an \emph{a priori} initial
$\mu$-entropy lower bound. Again, we may have to impose the control of the
initial diameter, as well as a uniform Ricci curvature lower bound for the
initial metric.

\subsection{Localization!}
\addtocontents{toc}{\protect\setcounter{tocdepth}{1}}
A viable principle in geometric analysis should see a reasonable localization.
We expect, in our future work
, that locally collapsing initial data with Ricci curvature bounded below should
also imply meaningful geometric structures on a positive-time slice. To indicate, we
notice that the renormalized Sobolev inequality (\ref{eqn: Sobolev}) we have
used is only the global version of the locally valid estimate (\ref{eqn:
local_Sobolev}); on the other hand, the recent foundational work on local
entropy by Bing Wang~\cite{Bing17}, provides the necessary technical tools that
pass the initial local Sobolev constant estimate to positive-time slices. We
would also like to mention the rencent work of Gang Tian and Zhenlei
Zhang~\cite{TZ18}, and the book by Qi S. Zhang~\cite{Zhang11} for related
efforts in this direction.


\begin{thebibliography}{99}
\bibliographystyle{plainnat}

\bibitem{Anderson92} Michael Anderson, The $L^2$ structure of moduli spaces
of Einstein metrics on 4-manifolds. \emph{Geom. Funct. Anal.}, 2 (1992), No. 1,
29-89.


\bibitem{BZ15a} Richard H. Bamler and Qi S. Zhang, Heat kernel and curvature
bounds in Ricci flows with bounded scalar curvature. \emph{Adv. Math.} 319
(2015), 396-450.

\bibitem{BZ15b} Richard H. Bamler and Qi S. Zhang, Heat kernel and curvature
bounds in Ricci flows with bounded scalar curvature - part II. \emph{Preprint},
arXiv:1506.03154, 2015.



\bibitem{Cheeger99} Jeff Cheeger, Differentiablity of Lipschitz functions on
metric measure spaces, \emph{Geom. Funct. Anal.} 9 (1999). No.3, 428-517.


\bibitem{ChCoI} Jeff Cheeger and Tobias Colding, On the structure of spaces
with Ricci curvature bounded below. I. \emph{J. Diff. Geom.} 46 (1997),
 406-480.

\bibitem{ChCoII} Jeff Cheeger and Tobias Colding, On the structure of spaces
with Ricci curvature bounded below. II. \emph{J. Diff. Geom.} 54 (2000), 13-35.

\bibitem{ChCoIII} Jeff Cheeger and Tobias Colding, On the structure of spaces
with Ricci curvature bounded below. III. \emph{J. Diff. Geom.} 54 (2000), 37-74.

\bibitem{CW12} Xiuxiong Chen and Bing Wang, Space of Ricci flows (I).
\emph{Comm. Pure Appl. Math.} Vol. LXV (2012), 1399-1457.

\bibitem{CW13} Xiuxiong Chen and Bing Wang, On the conditions to extend Ricci
flow (III). \emph{Int. Math. Res. Not.}  (2013), No.10, 2349-2367.


\bibitem{CW14} Xiuxiong Chen and Bing Wang, Space of Ricci flows (II).
\emph{Preprint}, arXiv:1405.6797, to appear in \emph{J. Diff. Geom.}

\bibitem{CW16} Xiuxiong Chen and Bing Wang, Remarks of weak-compactness along
Kahler Ricci flow. \emph{Preprint}, arXiv:1605.01374.

\bibitem{CY} Xiuxiong Chen and Fang Yuan, A note on Ricci flow with Ricci
curvature bounded below. \emph{J. Reine Angew. Math.} 726 (2017), 29-44.




\bibitem{CoNa11} Tobias H. Colding and Aaron Naber, Sharp H\"older continuity of
tangent cones for spaces with a lower Ricci curvature bound and applications.
\emph{Ann. of Math.} 176 (2012), 1173-1229.

\bibitem{Davies89} Edward Davies, \emph{Heat kernel and spectral theory}.
Cambridge University Press, 1989.



\bibitem{Fukaya87b} Kenji Fukaya, Collapsing of Riemannian manifolds and
eigenvalues of Laplace operator. \emph{Invent. Math.} 87 (1987), 517-527.





\bibitem{GromovLevy} Mikhail Gromov, Paul Levy's isoperimetric inequality.
\emph{Reprint},  www.ihes.fr/~gromov/PDF/11[33].pdf

\bibitem{Gromov_Metric} Mikhail Gromov, Metric structure for Riemannian and
non-Riemannian spaces. Modern Birkh\"auser Classics, 2006.


\bibitem{Ham93} Richar Hamilton, The formation of singularities in the Ricci
flow. \emph{Surveys in differential geometry}, Vol. II (Cambridge, MA, 1993),
Int. Press, Cambridge, MA (1995), 7-136.





\bibitem{KL08} Bruce Kleiner and John Lott, Notes on Perelman's papers.
\emph{Geom. Topol.} 12 (2008), 2587-2855.







\bibitem{Perelman02} Grisha Perelman, The entropy formula for the Ricci flow and
its geometric applications. \emph{Preprint}, arXiv:math/0211159.

\bibitem{SC92} Laurent Saloff-Coste, A note on Poincar\'e, Sobolev and
Harnack inequality. \emph{Int. Math. Res. Not.}  (1992), No.2, 27-38.


\bibitem{Simon09} Miles Simon, Ricci flow of almost non-negatively curved three
manifolds. \emph{J. Reine Angew. Math.} 630 (2009), 177-217.



\bibitem{TW15} Gang Tian and Bing Wang, On the structure of almost
Einstein manifolds. \emph{J. Amer. Math. Soc.} 28 (2015), no. 4, 1169-1209.

\bibitem{TZ18} Gang Tian and Zhenlei Zhang, Relative volume comparison of Ricci
flow and its applications. \emph{Preprint}, arXiv:1802.09506.

\bibitem{Topping05} Peter Topping, Diameter control under Ricci flow.
\emph{Comm. Anal. Geom.} 13 (2005), 1039-1055.



\bibitem{Bing17} Bing Wang, The local entropy along Ricci flow ---  Part A: the
no-local-collapsing theorems. To appear in \emph{Camb. J.
Math.}, arXiv:1706.08485.


\bibitem{Ye15} Rugang Ye, The logarithmic Sobolev and Sobolev inequalities along
the Ricci flow. \emph{Commun. Math. Stat.} 3, Issue 1 (2015), 1-36.

\bibitem{Zhang06} Qi S. Zhang, Some gradient estimates for the heat kernel
equation on domains and for an equation by Perelman. \emph{Int. Math. Res.
Not.} (2006), Art. ID 92314, 39 pp.

\bibitem{Zhang07} Qi S. Zhang, A uniform Sobolev inequality under Ricci flow.
\emph{Int. Math. Res. Not.} 17 (2007), Art. ID rnm056, 17 pp.



\bibitem{Zhang07Err} Qi S. Zhang, Erratum to: A uniform Sobolev inequality
under Ricci flow. \emph{Int. Math. Res. Not.} (2007).

\bibitem{Zhang07Add} Qi S. Zhang, Addendum to: A uniform Sobolev inequality
under Ricci flow. \emph{Int. Math. Res. Not.} (2007).

\bibitem{Zhang11} Qi S. Zhang, \emph{Sobolev inequlities, heat kernels under
Ricci flow, and the Poincar\'e conjecture}. CRC Press, 2011. ISBN
978-1-4398-3459-6.

\bibitem{Zhang12} Qi S. Zhang, Bounds on volume growth of geodesic balls
under Ricci flow. \emph{Math. Res. Lett.} 19 (2012), No. 1, 245-253.

\bibitem{Zhang14} Qi S. Zhang, On the question of diameter bounds in Ricci flow.
\emph{Illinois J. Math.} 58 (2014), No. 1, 113-123.

\bibitem{Zhu16} Meng Zhu, Davies type estimate and the heat kernel bound under
the Ricci flow. \emph{Trans. Amer. Math. Soc.} 368 (2016), No. 3, 1663-1680.

\end{thebibliography}
\end{document}